\documentclass[11pt,xcolor=dvipsnames]{amsart}

\usepackage[colorlinks=true, pdfstartview=FitV, linkcolor=blue, 
citecolor=blue]{hyperref}

\usepackage[dvipsnames]{xcolor}
\usepackage{amssymb,amsmath,amscd}

\usepackage{graphicx}
\usepackage{bbm}
\usepackage{a4wide}
\usepackage{enumitem}
\usepackage{verbatim}
\usepackage{float}
\usepackage{accents}
\usepackage{stackrel}
\usepackage{faktor}
\usepackage[dvipsnames]{epsfig}
\usepackage[dvipsnames]{psfrag}  
\usepackage{pstricks}
\usepackage{pstricks-add}
\usepackage{etex} 

\usepackage{pgfplots}
\usepackage{eufrak}
\usepackage{array}
\usepackage{cancel}

\newcolumntype{C}[1]{>{\centering\arraybackslash}p{#1}}

\pgfplotsset{width=6cm, compat = 1.15}

\theoremstyle{plain}
\newtheorem{theorem}{Theorem}[section]
\newtheorem{prop}[theorem]{Proposition}
\newtheorem{lemma}[theorem]{Lemma}
\newtheorem{coro}[theorem]{Corollary}

\theoremstyle{definition}
\newtheorem{remark}[theorem]{Remark}
\newtheorem{example}[theorem]{Example}
\newtheorem{definition}[theorem]{Definition}

\renewcommand{\geq}{\geqslant}

\renewcommand{\leq}{\leqslant}

\newcommand{\ts}{\hspace{0.5pt}}
\newcommand{\nts}{\hspace{-0.5pt}}

\newcommand{\RR}{\mathbb{R}\ts}
\newcommand{\PP}{\mathbb{P}\ts}

\newcommand{\NN}{\mathbb{N}}

\newcommand{\EE}{\mathbb{E}}
\newcommand{\cA}{\mathcal{A}}
\newcommand{\cB}{\mathcal{B}}

\newcommand{\cH}{\mathcal{H}}
\newcommand{\cI}{\mathcal{I}}
\newcommand{\cM}{\mathcal{M}}

\newcommand{\cL}{\mathcal{L}}

\newcommand{\cG}{\mathcal{G}}

\newcommand{\cP}{\mathcal{P}}
\newcommand{\cQ}{\mathcal{Q}}

\newcommand{\cX}{\mathcal{X}}
\newcommand{\cY}{\mathcal{Y}}

\newcommand{\one}{\mathbbm{1}}

\newcommand{\udo}[1]{\underaccent{$\text{.}$}{#1\ts}\nts}

\newcommand{\fA}{\mathfrak{A}}
\newcommand{\fx}{\mathfrak{x}}
\newcommand{\fy}{\mathfrak{y}}
\newcommand{\fz}{\mathfrak{z}}
\newcommand{\fX}{\mathfrak{X}}
\newcommand{\fY}{\mathfrak{Y}}
\newcommand{\fZ}{\mathfrak{Z}}
\newcommand{\Psireco}{\Psi_{\text{rec}}}
\newcommand{\Psisel}{\Psi_{\text{sel}}}
\newcommand{\Psiseltilde}{\tilde{\Psi}_{\text{sel}}}
\newcommand{\Psirecotilde}{\tilde{\Psi}_{\text{rec}}}

\newcommand{\bs}{\boldsymbol}

\newcommand{\1}{\mathbf{1}}

\newcommand{\owl}[1]{\,\overline{\! #1 \!}\,}

\newcommand{\ee}{\mathrm{e}}

\newcommand{\dd}{\, \mathrm{d}}

\newcommand{\proj}{P}

\newcommand{\myfrac}[2]{\frac{\raisebox{-2pt}{$#1$}}
      {\raisebox{0.5pt}{$#2$}}}

\definecolor{gre}{rgb}{.06,.49,0.03} 
\DeclareMathOperator{\ASG}{ASG}
\DeclareMathOperator{\ASRG}{ASRG}
\DeclareMathOperator{\cASRG}{cASRG}

\DeclareMathOperator{\card}{card}

\DeclareMathOperator{\geom}{Geom}
\DeclareMathOperator{\exponential}{Exp}
\DeclareMathOperator{\negbin}{NegBin}
\DeclareMathOperator{\id}{id}

\DeclareMathOperator*{\bigtim}{\raisebox{-2pt}{\scalebox{2.}{$\times$}}}
\DeclareMathOperator*{\bigboxplus}{\raisebox{-4pt}{\scalebox{2.}{$\boxplus$}}}
\DeclareMathOperator*{\bigboxtimes}{\raisebox{-4pt}{\scalebox{2.}{$\boxtimes$}}}
\DeclareMathOperator{\smallboxplus}{\raisebox{0pt}{\scalebox{0.8}{$\boxplus\ts$}}}
\DeclareMathOperator{\smallboxminus}{\raisebox{0pt}{\scalebox{0.8}{$\boxminus\ts$}}}
\DeclareMathOperator{\smallboxtimes}{\raisebox{0pt}{\scalebox{0.8}{$\boxtimes\ts$}}}

\newcommand{\defeq}{\mathrel{\mathop:}=}
\newcommand{\eqdef}{=\mathrel{\mathop:}}

\makeatletter
\newcommand*{\bigs}[1]{{\hbox{$\left#1\vbox to36\p@{}\right.\n@space$}}}
\makeatother

\definecolor{gre}{rgb}{.00,.74,0.00} 

\usetikzlibrary{arrows.meta}

\begin{document}

\title[Ancestral lines under selection and recombination] {Solving the selection-recombination equation: \\[2mm]
Ancestral lines and dual processes}

\author{Frederic Alberti and Ellen Baake}

\address{\{Faculty of Mathematics, Faculty of Technology\}, Bielefeld University, 
\hspace*{\parindent}Postbox 100131, 33501 Bielefeld, Germany}
\email{\{falberti,ebaake\}@math.uni-bielefeld.de}

\begin{abstract}
The deterministic selection-recombination equation describes the evolution of the genetic type composition of a population  under  selection and recombination in a law of large numbers regime. So far, an explicit solution has seemed out of reach; only in the special case of three sites with selection acting on one of them has an approximate solution been found, but without an obvious path to generalisation.
We use both an analytical and a probabilistic, genealogical approach for the case of an \emph{arbitrary} number of neutral sites  linked to one selected site. This leads to a recursive integral representation of the solution.
Starting from a variant of the ancestral selection-recombination graph, we develop an efficient genealogical structure, which may, equivalently, be represented as a weighted partitioning process, a family of Yule processes with initiation and resetting, and a family of initiation processes. We prove them to be dual to the solution of the differential equation forward in time and thus obtain a stochastic representation of the deterministic solution, along with the Markov semigroup in closed form.

\end{abstract}

\maketitle


\bigskip


\section{Introduction}
The \emph{recombination equation} is a well-known nonlinear system of ordinary differential equations
from mathematical population genetics (see \cite{Buerger} for the general background),
which describes the evolution of the genetic composition of a
population evolving under 
recombination. The genetic composition
is identified with a probability measure on a space of
sequences of finite length; 
and recombination is the genetic mechanism
by which, loosely speaking, two parent individuals create the mixed
sequence of their offspring during sexual reproduction, by means of one or several crossovers between the parental sequences.   Elucidating the underlying structure and finding solutions
was  a challenge for a century, namely since the first studies by Jennings \cite{Jennings} in 1917  and
Robbins  \cite{Robbins} in 1918.  The matter finally became simple and transparent when the corresponding stochastic backward (or dual) process was considered, which describes how the genetic material of an individual from the current population is partitioned across an increasing number of ancestors when the lines of descent are traced back into the past \cite{reco,haldane}. This gives rise to a Markov process on the set of partitions of the set of sequence sites; namely, a variant of the \emph{ancestral recombination graph} \cite{hudson,griffithsmarjoram96,griffithsmarjoram97,JenkinsFearnheadSong,BhaskarSong,LambertSchertzer}, see also \cite[Ch.~3.4]{durrett}.   With its help, one obtains a stochastic representation of the solution of the (deterministic) recombination equation, and  a recursive solution of the Markov semigroup, see \cite{haldane,reco}, and \cite{recoreview} for a review. Furthermore, it provides  the deeper reason for the underlying linear structure, which had been observed previously in the context of genetic algebras  \cite{HaleRingwood,Lyubich}. The recombination equation may therefore be considered solved.

We now take the next step and attack the \emph{selection-recombination equation}, which describes evolution under the joint action of recombination \emph{and selection}, where selection means that fit individuals flourish at the expense of less fit ones.  The selection-recombination equation first appeared  in a paper by Kimura   \cite{Kimura} in 1956. This differential equation, as well as the analogous discrete dynamical system,  has since been studied intensely for a large variety of selection and recombination mechanisms,  see, for example, \cite{Karlin,KirzhnerLyubich}, \cite[Chs. 9.5,9.6]{Lyubich}, as well as  \cite[Ch.~II]{Buerger} and \cite[Chs.~7,8]{Christiansen} for  comprehensive reviews.  Most research has focussed on the long-term behaviour, which can be complex and may display subtle and counterintuitive dependence on the parameters; in particular,  Hopf bifurcations and stable limit cycles may occur \cite{Hastings,Akin}. Much research has been devoted to the case where recombination is much faster than selection, so that time-scale separation applies and the  dynamics is confined to a submanifold, see \cite{Nagylaki,PontzHofbauerBuerger}. 

While a large body of knowledge has accumulated on the long-term behaviour,  explicit solutions have seemed out of reach even in the simplest nontrivial cases. Indeed, the monograph \cite{Akinbook} by Akin on the differential geometry of population genetics  starts with the sentence `The differential equations which model the action of selection and recombination are nonlinear equations which are impossible to solve explicitly.'  The only situation where an \emph{approximate} solution has been found   is  a sequence of length three in a two-letter alphabet, where only one of the sites is under selection, and recombination  involves one breakpoint (or \emph{crossover}) at a time between the parental sequences (Stephan, Song, and Langley 2006 \cite{stephansonglangley}). The approximation (in terms of special functions) seems  sufficiently precise, but the derivation is cumbersome and does not reveal the underlying mathematical structure; in particular, it does not convey any hope for a  generalisation beyond three sites.

The goal of this article is to reconsider the selection-recombination equation with one selected site and single crossovers, to provide a  systematic and transparent approach that also generalises to an \emph{arbitrary} number of sites, and to  establish an \emph{exact} solution via a recursion. We do this in two ways that complement each other: firstly, we solve the differential equation in the usual (forward) direction of time by analytic methods with a slight algebraic flavour. Secondly, we extend the probabilistic approach used in \cite{reco,haldane} for the pure recombination equation by tracing back the (potentially) ancestral lines of individuals in the current population, this time by a variant of the \emph{ancestral selection-recombination graph} \cite{DonnellyKurtz,LessardKermany,BossertPfaffelhuber}. This gives rise to a Markov process on the set of \emph{weighted} partitions of the set of sequence sites, dual to the selection-recombination equation. The corresponding Markov semigroup is available in closed form, and the resulting stochastic representation  yields deep  insight into the genealogical content of the solution of the differential equation. Moreover, it gives access to the long-term behaviour. 

\smallskip

The paper is organised as follows.  Sections~\ref{sec:selreco} and \ref{sec:Moran} introduce the selection-recombination equation, both in its own right and in terms of a dynamical law of large numbers of the corresponding Moran model, an interactive particle system that describes a \emph{finite} population under selection and recombination. 
 A recursive integral representation of the solution
 is given in  Section~\ref{sec:solution}. In Sections~\ref{sec:asrg} and~\ref{sec:interlude}, we construct the stochastic process backward in time and provide the genealogical argument behind our recursion. The corresponding dual process is formulated, and the formal duality result is proved,  in Section~\ref{sec:duality}. Finally,  the solution is presented in Section~\ref{sec:asymptotics} in closed form, and its long-term behaviour investigated.
  In the Appendix, we discuss  \emph{marginalisation consistency}, which describes the forward dynamics when only a subset of the sites  is considered. This is a fairly obvious, but nevertheless  powerful, property in the case without selection \cite{reco}. In the presence of selection, however, it is more subtle and only true for certain subsets, but all the more interesting.

\section{The selection-recombination equation}
\label{sec:selreco}
We model the distribution of the genetic types in a sufficiently large (hence effectively infinite)  population under  selection and  recombination. The \emph{genetic type} of an individual is represented by a sequence $x$ on the set  $S \defeq \{1, \ldots, n \}$ of sites and in the \emph{type space} 
\begin{equation}\label{typespace}
X \defeq \displaystyle{\bigtim_{i \in S} X_i = X_1 \times \ldots \times X_n   \quad \text{with } X_i = \{0,1\} }.
\end{equation}  
The \emph{type distribution} is identified with a probability measure $\omega \in \cP(X)$ on the type space, where $\cP(X)$ denotes the set of all probability measures on $X$. More generally, we define $\cM(X)$ to be the set of all finite signed measures on $X$. 

It is convenient to think of $\omega$ as an element of the vector space
\begin{equation}\label{typetensor}
V = V^{}_S \defeq \bigotimes_{i \in S} V_i = V_1 \otimes \ldots \otimes V_n,
\end{equation}
where each $V_i$ is a copy of $\RR^2$ and $\bigotimes$ denotes the tensor product of vector spaces. Here, a vector $v_i \in \RR^2$ of the form $v_i = (p_i,1-p_i)^T$ for some $p_i \in [0,1]$ is identified with the probability distribution $p_i \ts \delta_0 + (1 - p_i) \delta_1$ on $X_i$ with $\delta_0^{}$ ($\delta_1^{}$) denoting the point measure on $0$ (on $1$). The elementary tensors $v_1 \otimes \ldots \otimes v_n$ correspond to products of one-dimensional marginals. Hence, it is easy to see that the tensor product of vector spaces in \eqref{typetensor} provides an equivalent description of $\cM(X)$. 

For a subset $A \subseteq S$, we define the canonical projection
\begin{equation}\label{defproj}
\pi^{}_{A} : \; X \xrightarrow{\quad}  \bigtim_{i \in A} X_i \eqdef X_A, \quad x \mapsto (x_i)^{}_{i \in A} \eqdef x^{}_{\!A}.
\end{equation}
The push-forward of any  $\nu \in \cM(X)$ by $\pi_{\! A}^{}$ is denoted by $\pi_{\! A}^{}. \nu$, which we abbreviate by $\nu^A_{}$. Thus, $\nu^A$ is the marginal measure (or marginal distribution, if $\nu$ is a probability measure) with respect to the sites in $A$. More explicitly,
\begin{equation}\label{urbild}
\nu^A_{} (E) = \nu \big (\pi_{\! A}^{-1} (E) \big ) \quad \text{for all } E \subseteq  X_A.
\end{equation}

Note that $X_S=X$ and $\nu^S_{} =\nu$. Moreover, $X_\varnothing$ is the set with the single element $e$, which we think of as the empty  sequence. Thus, $\cM(X_\varnothing)$ is isomorphic to $\RR$, and
\begin{equation}\label{not2}
\nu^\varnothing = \nu(X).
\end{equation}
Furthermore, we write $\alpha \otimes \nu$ or $\nu \otimes \alpha$ instead of $\alpha \nu$, for all $\nu \in \cM(X_A)$ and $\alpha \in \RR$, in line with the usual identification of the empty tensor product with the base field.
In particular, if $\nu \in \cP(X_\varnothing)$, one has $\nu=1$, and the above convention just means to omit such factors from products. 
Later, we need to project not only from $X$, but also from factors $X_A$. In order to keep the notation simple, all  these projections will be denoted by the same letter $\pi$. In particular, we will write, for any two subsets $A \subseteq S$ and $B \subseteq S$ and any $\nu \in \cM(X)$,
\begin{equation}\label{not1}
\pi^{}_{\! A} . (\pi_{\! B}^{} . \nu ) = \pi^{}_{\! B} . ( \pi^{}_{\! A} . \nu) = \pi^{}_{\! A \cap B} . \nu.
\end{equation}
In line with Eq.~\eqref{not2}, this implies  that
$\pi^{}_{\! A} . \nu = \nu(X_B) \in \RR$
for any finite signed measure $\nu \in \cM(X_B)$ and $A \subseteq S$ with $A \cap B = \varnothing$.

To describe the action of  selection, we first  fix a site $1 \leq i_{\ast} \leq n$, which we will refer to as the \emph{selected site}. An individual of type $x \in X$ is deemed to be \emph{fit} or \emph{of beneficial type} if $x_{i_\ast} = 0$ and \emph{unfit} or \emph{of deleterious type} otherwise, regardless of the letters at all other sites. We also introduce 
\begin{equation} \label{fitproportion}
f(\nu) \defeq \nu \big (\pi_{i_\ast}^{-1}(0) \big ) = \nu_{}^{\{i_\ast\}}(0)
\end{equation}
for the proportion of fit individuals in a population with type distribution $\nu$, and 
the \emph{selection operator} \mbox{$F : \cP(X) \to \cP(X)$}  via
\begin{equation}\label{fitnessoperator}
F(\nu)(x) =  (1-x_{i_*}) \ts \nu(x).
\end{equation}
Interpreting the type distribution as an element of $V$ as given in \eqref{typetensor}, the selection operator can also be written in tensor notation as 
\begin{equation} \label{tensor}
F = \proj^{}_{{i_\ast}} \otimes \id_{{S^*}}.
\end{equation}
Here 
\[
\proj \defeq \begin{pmatrix} 1 & 0 \\ 0 & 0 \end{pmatrix},
\]
the subscripts indicate the site(s) at which the matrices act, 
and we set $S^* := S \setminus i_*$, where we use the shorthand $S \setminus i_*$ for $S \setminus \{i_*\}$ (note that $\card(S^*)=n-1$). 
In words, $F$ is the canonical projection to the subspace spanned by all elements of the form
\begin{equation*}
\begin{pmatrix}1 \\ 0 \end{pmatrix} \otimes v \; \text{ with } v \in  \bigotimes_{i \in S^*} \!  V_i,
\end{equation*}
and we recall that $(1,0)^T$ and $(0,1)^T$ correspond to the point measures $\delta_0$ and $\delta_1$ on $X_{i_*}$.
Furthermore, we define 
\begin{equation}\label{omegacond1}
b(\nu) \defeq  \frac{F \nu}{f(\nu)} \quad \text{if } f(\nu) \neq 0, \quad 
\text{complemented by } \; b(\nu) \defeq \nu \quad  \text{if  } \; f(\nu) = 0,
\end{equation}
and
\begin{equation}\label{omegacond2}
 d(\nu) =  \frac{(1-F) \nu}{1-f(\nu)} \quad \text{if } f(\nu) \neq 1,
\quad 
\text{complemented by } \; d(\nu) \defeq \nu \quad  \text{if  } \; f(\nu) = 1.
\end{equation}
We will also write $F \nu$ instead of $F(\nu)$ where there is no risk of confusion.
The measure $b(\nu)$ (the measure $d(\nu)$) is the type distribution in the beneficial (deleterious) subpopulation.

Selection now works as follows. Unfit individuals reproduce at rate $1$, while fit individuals reproduce at  rate $1 + s$, $s > 0$. Put differently, every individual, regardless of its type, has the neutral reproduction rate $1$, while the fit individuals have an additional (selective) rate $s$. The net effect of this  is that, in each infinitesimal time interval of length $\dd t$, an infinitesimal portion $s f(\omega^{}_t) \dd t$ of $\omega^{}_t$ is replaced by $b(\omega^{}_t)$. That is, the dynamics of the type distribution of our population under selection alone can be described by the ordinary differential equation
\begin{equation}\label{selection var1}
\dot{\omega_t} = s f(\omega^{}_t) \big ( b(\omega^{}_t) - \omega^{}_t \big ).
\end{equation}
With the notation \eqref{omegacond1}, Eq.~\eqref{selection var1}  turns into the \emph{deterministic selection equation}
\begin{equation}\label{pure_selection}
\dot{ \omega^{}_t } =  s \big (F - f(\omega^{}_t) \big) \omega^{}_t \eqdef \Psisel(\omega^{}_t).
\end{equation}
We will sometimes speak of $s$ as the \emph{selection intensity}.

Next, we describe the action of single-crossover recombination. To this end, it is vital to introduce the following partial order on $S$.
\begin{definition}\label{porder}
For two sites $i,j \in S$, we say that $i$ precedes $j$, or $i \preccurlyeq j$, if either $i_\ast \leq i \leq j$ or $i_\ast \geq i \geq j$. We write $i \prec j$ if $i \preccurlyeq j$ and $i \neq j$. We furthermore define the $i$-\emph{tail} as the set
\begin{equation*}
D_i \defeq \{j \in S  :  i \preccurlyeq j\}
\end{equation*}  
of all sites that succeed $i$, including $i$ itself. We define the $i$-\emph{head} $C_i$ as the complement of the $i$-tail, $C_i \defeq S \setminus D_i = \owl{D_i}$ (throughout, the overbar will denote the complement with respect to $S$); see Figure \ref{headsntails}. Note that $D_{i_*}=S$ and $C_{i_*} = \varnothing$. 
Finally, if $i \neq i_\ast$, we denote by $\overleftarrow i$ the \emph{predecessor} of $i$; that is, the maximal $j \in S$ with $j \prec i$ (note that  $\overleftarrow i = i_*$ is possible).
\end{definition}

\begin{figure}
\includegraphics[width = 0.85 \textwidth]{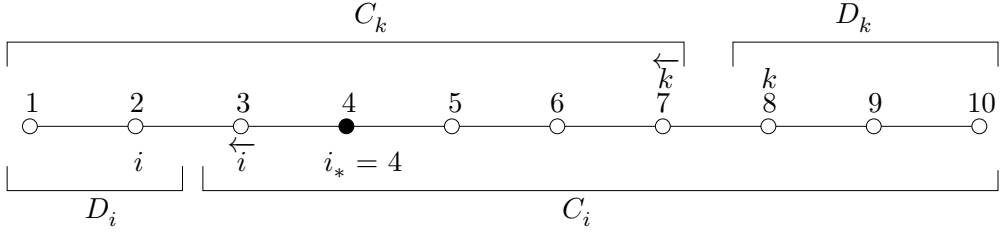}
\caption{\label{headsntails} A sequence of length 10 with selected site, and two instances of predecessor, head, and tail; see the text for more.
}
\end{figure}

\begin{remark} \label{rem:awkward}
  \begin{enumerate}
  \item Definition~\ref{porder} may appear awkward in that $i_* \in D_{i_*}$ but $i_* \in C_i$ for $i \in S^*$. However, it will become clear in Section~\ref{sec:duality} why this is exactly how it must be.
    \item   In the limiting case $s=0$, we may single out any site as the selected one; say $i_*=n$, so that $D_i=[1:i]$ and $C_i=[i+1:n]$, where $[a:b]$ for $a,b \in S$ denotes the interval $[a,a+1, \ldots, b-1,b]$, which is  empty  if $b < a$.
  \hfill $\diamondsuit$
    \end{enumerate}
\end{remark}

For $i \in S^*$,  we now define the \emph{recombinator} $R_i : \cP(X) \to \cP(X)$  by
\begin{equation}\label{defrecombinator}
R_i (\nu) \defeq  \nu_{}^{C_i} \otimes \nu_{}^{D_i},
\end{equation}
with the notation of  \eqref{defproj} and \eqref{urbild}; we will also write $R_i \nu$ instead of $R_i (\nu)$. Then, the dynamics under the influence of single-crossover recombination is captured by the \emph{deterministic recombination equation}
\begin{equation}\label{recoeq}
\dot{\omega_t} = \sum_{i  \in S^*} \varrho_i (R_i  - \id) \omega_t \eqdef \Psireco(\omega_t)
\end{equation}
with \emph{recombination rates} $\varrho_i \geqslant 0$ for $i \in S^*$; for consistency, we set $\varrho_{i_*}:= 0$.

On an intuitive level, Eq.~\eqref{recoeq} means that during each infinitesimal time interval of length $\dd t$ and for every  $i \in S^*$, an infinitesimal portion of size $\varrho_i \dd t$ of the population is killed off and replaced by the offspring of two randomly chosen parent individuals of types $x = (x_1,\ldots,x_n)$ and $y = (y_1,\ldots,y_n)$ (which occur in the current population with frequencies $\omega_t(x)$ and $\omega_t(y)$, respectively); the offspring then has  type $( x^{}_{C_i},y^{}_{D_i})$.  This means that, for $i<i_*$ ($i>i_*$), a single-crossover event takes place between sites $i$ and $i+1$ (sites $i-1$ and $i$); in any case,  we say that recombination happens at site $i$. This way, we address the \emph{links} between neighbouring sites, as in \cite{fehler}, but in a way that depends on the location of the selected site.

Occasionally (cf.~Section 6), it will be handy to employ a more general notion of  recombinators in terms of partitions; a \emph{partition} of $S$ is a set $\cA$ of nonempty, disjoint subsets of $S$ that exhaust $S$. We will refer to the elements of $\cA$ as \emph{blocks}. We denote by $P(S)$ the set of partitions of $S$ (not to be confused with $\cP(X)$, the set of probability measures on $X$). A partition is called an \emph{interval} (or \emph{ordered}) \emph{partition} if all its blocks are intervals, that is, consist of contiguous numbers. Given an (arbitrary) partition $\cA$ of $S$ and a nonempty subset $U \subseteq S$, we define by 
$\cA|_U \defeq \{U \cap A : A \in \cA \} \setminus \{ \varnothing \}$ 
the partition induced by $\cA$ on $U$. 
Generalising Eq.~ \eqref{defrecombinator}, we define, for an arbitrary partition $\cA$ of $S$,
\begin{equation}\label{generalrecombinator}
\tilde R_\cA^{} (\nu) \defeq \bigotimes_{A \in \cA} \nu_{}^A.
\end{equation}
Clearly, $R_i = \tilde R_{\{C_i,D_i\}}$ for $i \in S^*$. The formulation in terms of partitions is natural because it describes how the offspring sequence is pieced together from the parental sequences. 
We refer the interested reader to \cite{reco,haldane} and  the recent review \cite{recoreview} for a comprehensive discussion of the properties of $\tilde R_\cA$ and for the general recombination equation, which involves arbitrary partitions rather than single crossovers only. 

We now return to the single-crossover case and assume that selection and recombination   act independently of each other. Combining \eqref{pure_selection} and \eqref{recoeq}, we obtain the \emph{deterministic selection-recombination equation} (SRE)
\begin{equation}\label{main}
\dot{\omega^{}_t} = \Psi(\omega^{}_t), \quad \text{where } \Psi := \Psisel + \Psireco.
\end{equation}
The independence, as implied by the additivity, reflects the assumption that both selection and recombination are rare, so that one can neglect the possibility that recombination happens during  selective reproduction; see Remark~\ref{rem:sexual} below, and \cite{Hofbauer} for the worked argument in the analogous case of the selection-mutation equation.

We will throughout denote by $\omega \defeq (\omega_t^{})_{t \geqslant 0}$ the solution of Eq.~\eqref{main}.
Let us mention at this point that, for $i_* \in A \subseteq S$, the SRE is \emph{marginalisation consistent} in the sense that  the marginal  $\omega_{}^A \defeq (\omega_t^A)^{}_{t \geqslant 0}$ satisfies
\[
 \dot \omega_t^A =  \pi^{}_{\! A} . \Psisel(\omega^{}_t)  + \pi^{}_{\! A} . \Psireco(\omega^{}_t)  = \Psisel^A(\omega_t^A) + \Psireco^A(\omega_t^A)  
\]
with initial condition $\omega_0^A$,
for suitably defined $\Psisel^A$ and $\Psireco^A$; this will be laid out in Appendix A. 
Although this  is not essential for the core of the paper, it helps to understand the graphical constructions in Section~\ref{sec:asrg}, and is also of independent interest. There is no such consistency   for $i_* \notin A$, which is a source of difficulties and pitfalls in the selective case.

\section{The Moran model with selection and recombination} 
\label{sec:Moran}

To gain a better understanding of Eq.~\eqref{main} and to prepare for the genealogical arguments  in Section \ref{sec:asrg}, we briefly recall the Moran model with selection and recombination. This is a \emph{stochastic} model that describes selection and recombination in a \emph{finite} population, from which \eqref{main} is recovered via a dynamical law of large numbers (LLN). We will use the representation as an \emph{interacting particle system} (IPS). 
The Moran IPS is a Markov chain $(\Xi^{(N)}_t)_{t \geqslant 0}$ with state space $X^N$, the set of type configurations of a population of $N$ individuals, labelled by $1 \leqslant \alpha \leqslant N$.
Starting from some initial configuration $\big (\Xi_0^{(N)}(\alpha) \big )_{\alpha \in [1:N]}$, it evolves as follows.

\begin{itemize}
\item Every individual $\beta$  reproduces asexually at a fixed rate according to its fitness. That is, unfit individuals reproduce at rate $1$ whereas fit individuals reproduce at rate $1 + s$, where $s > 0$ is again the selection intensity. Upon reproduction, the single offspring inherits the parent's type and replaces a uniformly chosen individual $\alpha$ in the population (possibly  its own parent). We will realise the different reproduction rates of the two types by distinguishing between \emph{neutral reproduction events}, which happen at rate 1 to all individuals regardless of their type, and \emph{selective reproduction events}, which are additionally performed by fit individuals at rate $s$. This distinction is a crucial ingredient in the ancestral selection graph \cite{KundN}. 
\item
At rate $\varrho_i, \, i \in S^*$, every  individual $\beta$  reproduces sexually, choosing a partner $\gamma$ uniformly at random, possibly $\beta$ itself. (Biologically, this means that we include the possibility of \emph{selfing}.) The offspring is of type $\big (  \Xi^{(N)}_{C_i}(\beta), \Xi^{(N)}_{D_i}(\gamma) \big )$ and replaces another uniformly chosen individual $\alpha$, possibly one of its own parents. 
\end{itemize}
Formally, we can summarise the transitions in the Moran IPS, starting at $\xi \in X^N$, as follows.
\begin{equation*}
\begin{split}
& \xi \to \xi_{\text{neut}}^{\{\alpha,\beta\}} \text{ at rate $\myfrac{1}{N}$ for all } 1 \leq \alpha, \beta \leq N, \\
& \xi \to \xi_{\text{sel}}^{\{\alpha,\beta\}} \text{ at rate $\myfrac{s}{N}$ for all } 1 \leq \alpha,\beta \leq N, \quad \text{ and } \\
& \xi \to \xi_{\text{rec}}^{\{\alpha,\beta,\gamma,i\}} \text{ at rate $\frac{\varrho^{}_i}{N^2}$ for all } 1 \leq \alpha,\beta,\gamma \leq N \text{ and } i \in S^*,   
\end{split}
\end{equation*}
where, for $1 \leq \varepsilon \leq N$, the new state vectors explicitly read
\begin{equation}\label{xineu}
\xi_{\text{neut}}^{\{\alpha,\beta\}}(\varepsilon) = \begin{cases}
\xi(\beta), & \varepsilon = \alpha, \\
\xi(\varepsilon) , & \text{otherwise}, 
\end{cases} \quad
\xi_{\text{sel}}^{\{\alpha,\beta\}}(\varepsilon) = \begin{cases}
\xi(\beta), & \varepsilon = \alpha \text{ and } \xi^{}_{i_\ast}(\beta) = 0, \\
\xi(\varepsilon) , & \text{otherwise}, 
\end{cases}
\end{equation}
and
\begin{equation*}
\xi_{\text{rec}}^{\{\alpha,\beta,\gamma,i\}}(\varepsilon) \defeq  \begin{cases}
\big (  \xi_{C_i}(\beta),  \xi_{D_i}(\gamma) \big), & \varepsilon = \alpha, \\
\xi(\varepsilon) , & \text{otherwise}. \end{cases}
\end{equation*}

\begin{remark} \label{rem:sexual}
  The reader may wonder  why we include both sexual and asexual reproduction. However, the `asexual' reproduction events are actually  sexual ones in which no recombination has occurred; that is, $C = \varnothing$ and $D = S$, so the offspring is a full copy of the first parent, and the second parent is irrelevant. Selective reproduction never occurs together with recombination due to the independence built into the SRE.
  \hfill $\diamondsuit$
\end{remark}

Consider now the process $Z^{(N)} := (Z^{(N)}_t)_{t \geq 0}$, where  $Z^{(N)}_t$ is the empirical measure
\begin{equation*}
Z^{(N)}_t  \defeq \myfrac{1}{N} \sum_{\alpha= 1}^N \delta_{\ts \Xi^{(N)}_t(\alpha)}.
\end{equation*}
Proposition 3.1 in \cite{limitsel} in combination with Theorem 2.1 from \cite{limitreko} (see also \cite{baakeesserprobst}) shows that, as $N \to \infty$ without rescaling of parameters or time, the processes $Z^{(N)}$ converge almost surely locally uniformly to $\omega$, the solution of the deterministic SRE \eqref{main}. This is because the Moran models, indexed with population size, form a density-dependent family, for which a dynamical LLN applies; see \cite[Ch.~11]{ethierundkurtz}.

\begin{figure}
\includegraphics[width =  0.83\textwidth]{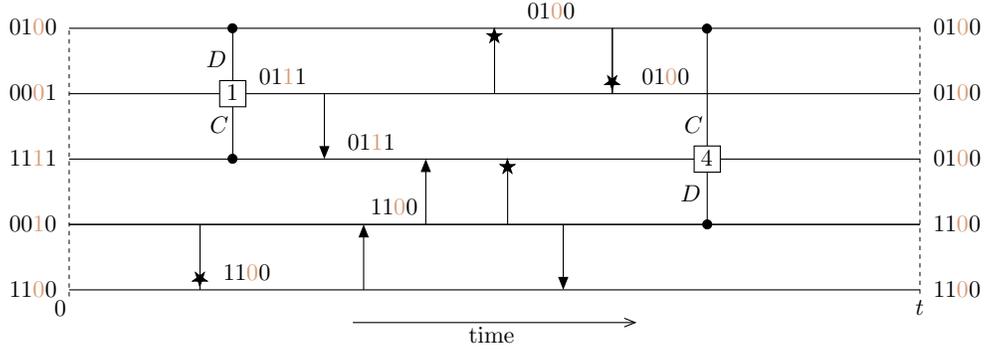}
\caption{Graphical representation of the Moran IPS. Time runs from left to right. Arrows corresponding to neutral reproduction events are depicted with normal arrowheads, selective arrows with star-shaped arrowheads; recombination events are symbolised by squares containing the recombination point, and arms connecting to the parents that contribute the head ($C$) and tail ($D$) segments. The selected site is marked in light brown. \label{MoranIPS}}
\end{figure}

For our purpose, it is particularly profitable to use the \emph{graphical representation} of the Moran IPS, see Figure~\ref{MoranIPS}.  Here, individuals are represented by horizontal lines, labelled $1 \leq \alpha \leq N$ from bottom to top, and reproduction events are depicted as arrows between the lines with the parent at the tail, the offspring at the tip, and the offspring replacing the individual at the target line (arrows pointing to their own tails  have no effect and are omitted).
In line with \eqref{xineu} and for reasons to become clear when taking the ancestral perspective in Section 5, we distinguish two types of arrows: \emph{neutral} arrows (with normal arrowheads), which appear between every ordered pair of lines at rate $1/N$ regardless of the types of the lines; and \emph{selective} arrows (with star-shaped arrowheads), which are laid down at rate $s/N$ between every ordered pair of lines, again regardless of types. 
Similarly, a recombination event in which the individual at line $\alpha$ is replaced by the joint offspring of lines $\beta$ and $\gamma$ is encoded as a square (on the $\alpha$-th line) with the recombination site $i$  inscribed. The square  has two arms connecting to $\beta$ and $\gamma$ and labelled $C$ and $D$, indicating that $\beta$ and $\gamma$  contribute the $i$-head and $i$-tail, respectively. These graphical elements appear at rate $\varrho^{}_i/N^2$ for every ordered triple of lines and every $i \in S^*$. If $\beta=\gamma$, the recombination event turns into a neutral reproduction event.

\begin{remark}
  In view of this graphical construction, another perspective on the transition rates in the Moran IPS is  natural. We can say that, with rates $\varrho_i$, each individual is replaced by the joint offspring of two uniformly chosen parents with the crossover point at site $i$. Likewise, at rate $1$, each individual is replaced by the offspring of a \emph{single} uniformly chosen parent; and with rate $s f(Z^{(N)}_t)$, it is replaced by the offspring of a parent  chosen uniformly from the subset of fit individuals.  This point of view will be particularly useful when looking back in time in Section \ref{sec:asrg}.
  \hfill $\diamondsuit$
\end{remark}

Using different kinds of arrows for the two types of reproduction events (rather than simply letting fit individuals shoot reproduction arrows at a faster rate)   allows for an \emph{untyped} construction of the Moran IPS. That is, we first lay down the graphical elements between the lines regardless of the types and only then  assign an initial type configuration. This type configuration is finally propagated forward in time under the rule that only individuals of beneficial type  use the selective arrows to place their offspring, while neutral arrows and recombination arms are used by all individuals, regardless of type.

\section{Recursive solution of the selection-recombination equation}
\label{sec:solution}

Our first main result will be a recursive solution of the SRE. The recursion  starts at $i_*$ and works along the site indices in agreement with the partial order of Definition~\ref{porder}. If the original  indices are used,  the recursion must be formulated individually for every choice of $i_*$; in particular, it looks quite different depending on whether $i_*$ is at one of the ends or in the interior of the sequence. To establish the recursion in a unified framework, we introduce a relabelling; 
let us fix a nondecreasing (in the sense of the partial order from Definition \ref{porder}) permutation  $(i_k)_{0 \leq k \leq n-1}$ of $S$ (cf. Fig.~\ref{headsntails_small}) and denote the corresponding heads and tails by upper indices, that is,  $C^{(k)} \defeq  C_{i_k}$ and $D^{(k)} \defeq  D_{i_k}$  (cf. Figure \ref{headsntails}). Note that $i_0 = i_\ast$, $D^{(0)} = S$ and $C^{(0)} = \varnothing$. Note also that this definition implies that for all $\ell \geq k$, one has either $D^{(\ell)} \subseteq D^{(k)}$ (if $\ell \succcurlyeq k$) or $D^{(\ell)} \subseteq C^{(k)}$ (if $\ell$ and $k$ are incomparable). Furthermore, we define $\varrho^{(k)} \defeq \varrho^{}_{i_k}$ and $R^{(k)} = R_{i_k}$ for $k > 0$. 

\begin{figure}
\includegraphics[width = 0.8 \textwidth]{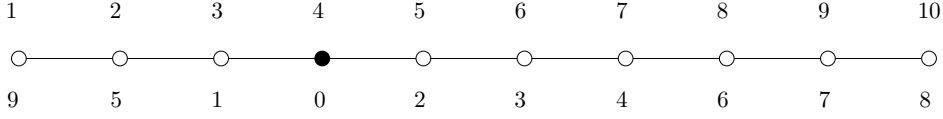}
\caption{\label{headsntails_small} A nondecreasing permutation of sites. The original labels of the sites, $1 \leq i \leq n$, are at the top;  below each site with label $i$, we have noted the corresponding $k$ for which $i_k = i$. }
\end{figure}

First, we recapitulate the solution of the pure selection equation, that is, we solve \eqref{main} in the special case that all recombination rates vanish. Then, in accordance with the labelling given by  $(i_k)_{1 \leq k \leq n-1}$, we will successively add sites at which we allow  recombination. We set the scene as follows.

\begin{definition}\label{hierarchy}
For $\varrho^{(1)}, \ldots, \varrho^{(n-1)}$ as above and every $k \in [0:n-1]$, we set 
\[
\Psireco^{(k)} \defeq \sum_{\ell=1}^{k} \varrho^{(\ell)} \big ( R^{(\ell)} - \id \big ), \quad \Psi^{(k)} \defeq \Psisel + \Psireco^{(k)}
\]
(with the usual convention that the empty sum is 0, whence $\Psireco^{(0)} = 0$ and $\Psi^{(0)} = \Psisel$). We then
define the SRE \emph{truncated at $k$}  as the differential equation  
\begin{equation*}
\dot \omega_t^{(k)} =  \Psi^{(k)} (\omega_t^{(k)}).
\end{equation*}
We understand $ ( \omega^{(k)} )_{0 \leq k \leq n-1}$ as the family of the corresponding solutions, all with the same initial condition $\omega^{}_0$. In particular, 
$\omega_{}^{(0)}$ is the solution of the pure selection equation \eqref{pure_selection}. We also define $\psi^{(k)}=(\psi_t^{(k)})^{}_{t \geq 0}$ as the flow semigroup associated to the  differential equation defined via $\Psi^{(k)}$.  In line with \eqref{main}, we have $\omega = \omega^{(n-1)}$ (which is to say $\omega_t^{} = \omega_t^{(n-1)}$ for all $t \geq 0$) and $\Psi = \Psi^{(n-1)}$, and we likewise set $\psi = \psi^{(n-1)}$. We  also write $\varphi$ instead of $\psi^{(0)}$.
\end{definition}

\begin{prop}\label{omeganull}
The solution of the pure selection equation  \eqref{pure_selection} with initial condition $\omega^{}_0 \in \cP(X)$ is given by
\begin{equation}\label{fitincreases}
\omega^{(0)}_t = \varphi^{}_t(\omega^{}_0) = \myfrac{\mathrm{e}^{st} F(\omega_0) + (1-F)(\omega_0)}{\mathrm{e}^{st} f(\omega_0) + 1 - f(\omega_0)}, \quad t \geq 0, 
\end{equation}
with $f$ and $F$ as given in \eqref{fitproportion} and \eqref{fitnessoperator}. In particular, 
\begin{equation}\label{fitproportion_sol}
f(\omega^{}_t) = \myfrac{\ee^{st} f(\omega^{}_0)}{\ee^{st} f(\omega^{}_0) + 1 - f(\omega^{}_0)}
\end{equation}
is  increasing over time and $\omega_t^{(0)} = \varphi^{}_t(\omega^{}_0)$ is a convex combination of the initial type distributions of the fit $($that is, beneficial\/$)$ and unfit $($that is, deleterious\/$)$ subpopulations introduced in Eqs. \eqref{omegacond1} and \eqref{omegacond2}, namely,
\begin{equation*}
\omega_t^{(0)} = f(\omega_t^{(0)}) b(\omega_0) + (1 - f(\omega_t^{(0)})) d(\omega_0).
\end{equation*}
 This implies in particular
  \begin{equation}\label{condunchanged}
 b \big (\varphi^{}_t(\omega^{}_0) \big ) =   b(\omega^{}_0) \quad
\text{and } \;
d \big (\varphi_t^{}(\omega^{}_0)) = d(\omega^{}_0).
\end{equation}

\end{prop}

\begin{proof}
A straightforward verification. To see Eq.~\eqref{condunchanged}, recall that $F$ is a projection and $b(\nu)$ is in the image of $F$, while $d(\nu)$ is in the image of $1-F$ for any $\nu \in \cP(X)$. 
\end{proof}

\begin{remark}
Eq.~\eqref{fitproportion_sol} generalises the well-known solution of the selection equation for a single site, which is simply a logistic equation, cf.~\cite[p.~198]{durrett}. Eq.~\eqref{condunchanged} reflects the plausible fact that, while the proportion of fit individuals increases at the cost of the unfit ones (as quantified in Eq.~\eqref{fitincreases}), the type composition \emph{within} the set of fit types remains unchanged, and likewise for the set of unfit types. \hfill $\diamondsuit$
\end{remark}

The main result in this section is the following recursion  for the  solutions of the (truncated) SREs.
\begin{theorem} \label{rekursion}
The family of solutions $ (\omega^{(k)})_{1 \leq k \leq n-1}$ of Definition \textnormal{\ref{hierarchy}} satisfies the recursion 
\begin{equation*}
\omega_t^{(k)} = \mathrm{e}^{-\varrho^{(k)} t} \omega_t^{(k-1)} + \pi_{C^{(k)}}. \omega_t^{(k-1)} \otimes \pi_{D^{(k)}}.  \int_0^t \varrho^{(k)} \mathrm{e}^{-\varrho^{(k)} \tau} \omega_\tau^{(k-1)} \dd \tau
\end{equation*}
for $1 \leq k \leq n-1$ and $t \geq 0$, where $\omega^{(0)}$ is the solution of the pure selection equation given in Proposition \textnormal{\ref{omeganull}}.
\end{theorem}
We will first give an analytic proof, followed, in the next section, by a genealogical proof  based on the ancestral selection-recombination graph ($\ASRG$); this will provide additional insight.

To deal with the nonlinearity of recombination and to exploit the underlying \emph{linear} structure (see \cite{haldane}) more efficiently, we now introduce a variant of the product of two measures that are defined on $X_A$ and $X_B$, where $A$ and $B$ need not be disjoint. Namely, given sets  $I,J \subseteq S$ and  finite signed measures $\nu_{I}^{},\nu_{J}^{}$  on $X_{I}^{}$ and $X_{J}^{}$, respectively, we define 
\begin{equation*}
\nu^{}_{I} \smallboxtimes \nu^{}_{J} \defeq (\pi^{}_{I \setminus J} . \nu^{}_{I}) \otimes \nu^{}_{J},
\end{equation*}
which is a finite signed measure on $X_{I \cup J}$ (recall that $\pi^{}_\varnothing . \nu = \nu(X_{I})$ for all finite signed measures $\nu$ on $X_{I}$, $I \subseteq S$).  
Note  that    $\nu^{}_I$ here means  any finite signed measure on  $X_I^{}$, whereas   $\nu^I$ stands for the specific  measure on $X_I^{}$ that is obtained from  $\nu$ on $X$ via $\nu^I = \pi^{}_I . \nu$.

\begin{prop}\label{algebraprops}
For  $I,J,K \subseteq S$ and  finite signed measures $\nu_{I}^{},\nu_{J}^{},\nu_{K}^{} $  on $X_{I}^{}$, $X_{J}^{}$, and $X_{K}^{}$, respectively, the operation $\smallboxtimes$ \/ has the following properties. 
\begin{enumerate}[label={\textnormal{(\roman*)}}] \itemsep=2pt
\item $(\nu^{}_I \smallboxtimes \nu^{}_J ) \smallboxtimes \nu^{}_K  =  \nu^{}_I \smallboxtimes \, (\nu^{}_J  \smallboxtimes \nu^{}_K)$ $($associativity\/$)$.
\item If $I \cap J = \varnothing$, we have $\nu^{}_{I} \smallboxtimes \nu^{}_{J}=\nu^{}_{I} \otimes \nu^{}_{J} = \nu^{}_{J} \smallboxtimes \nu^{}_{I}$ $($reduction to  product measure and commutativity\/$)$.
\item If $I \subseteq J$, then
$\nu^{}_I \smallboxtimes \nu^{}_J = \nu^{}_I(X_I^{}) \ts \nu_J^{}$ $($cancellation property\/$)$.
\end{enumerate}
\end{prop}
\begin{proof}
For associativity, note that
\begin{equation*}
\begin{split}
(\nu^{}_I \smallboxtimes \nu^{}_J ) \smallboxtimes \nu^{}_K  & = 
\big ( (  \pi^{}_{I\setminus J} . \nu^{}_I) \otimes \nu^{}_J  \big ) \smallboxtimes \nu^{}_K =
\big (   \pi^{}_{(I \cup  J)\setminus K} . (\pi^{}_{I \setminus J} .\nu^{}_I) \otimes \nu^{}_J  \big ) \otimes \nu^{}_K \\ &=
\pi^{}_{I \setminus (J \cup K)} . \nu^{}_I \otimes \pi^{}_{J \setminus K} . \nu^{}_J \otimes \nu^{}_K = \pi^{}_{I \setminus (J \cup K)} \otimes (\nu^{}_J \smallboxtimes \nu^{}_K ) = \nu^{}_I \smallboxtimes \, (\nu^{}_J  \smallboxtimes \nu^{}_K),
\end{split}
\end{equation*}
where we have used in the third step that $((I \cup J) \setminus K) \cap (I \setminus J) = I \setminus (J \cup K)$.

When $I \cap J = \varnothing$, one has
\[
 \nu^{}_I \smallboxtimes \nu^{}_J = \pi^{}_{I \setminus J}.\nu^{}_I \otimes \nu^{}_J = \pi^{}_{I }.\nu^{}_I \otimes \nu^{}_J = \nu^{}_I \otimes \nu^{}_J = \nu^{}_J \otimes \nu^{}_I,
\] 
which implies the claimed reduction to $\otimes$ and thus commutativity.
Finally, for $I \subseteq J$,
\[
\nu^{}_I \smallboxtimes \nu^{}_J = ( \pi^{}_{I \setminus J}  .  \nu^{}_I) \otimes \nu_J^{} = (\pi^{}_\varnothing . \nu^{}_I) \otimes \nu^{}_J = \nu^{}_I(X_I) \nu_J^{}
\]
establishes the cancellation property.
\end{proof}

Under the conditions of Proposition~\ref{algebraprops}, we now denote by $\nu_{J}^{} \smallboxplus \nu_{K}^{}$  the formal sum of $\nu_{J}^{}$ and $\nu_{K}^{}$ (and use $\smallboxminus$ for the corresponding formal difference). Note that the formal sum turns into a proper sum (and hence $\smallboxplus$ reduces $+$)  when  $J=K$. Furthermore, we define
\begin{equation}\label{distributiv}
  \nu^{}_I \smallboxtimes \, (\nu_{J}^{} \smallboxplus \nu_{K}^{}) \defeq 
  ( \nu^{}_I \smallboxtimes \nu_{J}^{} ) \smallboxplus \, (\nu^{}_I \smallboxtimes \nu_{K}^{}).
\end{equation}
Clearly, the right-hand side reduces to a proper sum when $I \cup J=I \cup K$.

Generalising the formal sum above, let $\fA(X_U)$  be the real vector space of formal sums 
\begin{equation*}
\nu \defeq \lambda_1 \nu_{U_1}^{} \smallboxplus \ldots  \smallboxplus \lambda_q \nu_{U_q}^{},
\end{equation*}
where $q \in \NN$, $\lambda_1, \ldots, \lambda_q \in \RR$,  $U_1,\ldots,U_q \subseteq U \subseteq S$, and $\nu_{U_1}^{},\ldots,\nu_{U_q}^{}$ are finite signed measures on $X_{U_1}^{},\ldots,X_{U_q}^{}$, respectively. We also write $\nu(X_U) \defeq \sum_{i=1}^q \lambda_i \nu^{}_{U_i}(X_{U_i})$.

\begin{remark}
If one extends  the definition of $\smallboxtimes$ canonically to all of $\fA(X_U)$ (recalling that the projections are linear), $\big (\fA(X_U),\smallboxtimes, \smallboxplus \big )$ becomes an associative, unital algebra with neutral element $\1$, the measure with weight 1 on $X_\varnothing$. 
Recall that, when multiplying $\nu \in  \fA(X_I)$ and $\mu \in \fA(X_J)$ for disjoint $I$ and $J$, the multiplication $\smallboxtimes$ agrees with the measure product $\otimes$. \hfill $\diamondsuit$
\end{remark}

Now, we can rewrite $\Psireco^{(k)}$ of  Definition~\ref{hierarchy} as
\begin{equation}\label{psireconew}
\Psireco^{(k)} \big (\nu \big ) =  \nu \smallboxtimes \Big ( \bigboxplus_{\ell=1}^{k}  \varrho^{(\ell)} \big ( \pi^{}_{D^{}_\ell} . \nu - \1 \big) \Big);
\end{equation}
note that the right-hand side  indeed reduces to a proper (rather than a formal) sum of measures via \eqref{distributiv}, because every summand is a measure on $X_S$.

We shall see later that, when combined with selection,  this representation is superior to the use of recombinators because it nicely brings out the recursive structure; this will streamline calculations and naturally connect to the graphical construction.
The fact that the head alone determines the fitness of an individual manifests itself in the right-multiplicativity of $\Psisel$ and its associated flow $\varphi$ (compare Definition \ref{hierarchy}), as we shall see next.
\begin{lemma}\label{equivariance}
For all $\mu \in \cP(X)$ and all $\nu \in \fA(X_{S^*})$, 
\begin{equation*}
F(\mu \smallboxtimes \nu) = F(\mu) \smallboxtimes \nu.
\end{equation*}
If, in addition, $\nu(X_{S^*}) = 1$, one has
\[
 \Psi_{\rm{sel}} (\mu \smallboxtimes \nu) = \Psi_{\rm{sel}}(\mu) \smallboxtimes \nu
\]
and therefore
$
\varphi_t^{}(\mu\smallboxtimes \nu) = \varphi_t^{}(\mu) \smallboxtimes \nu
$
for every $t \geq 0$.

\end{lemma}
\begin{proof}
To keep the notation simple, we assume $U_1,U_2 \subseteq S^*$ and $\nu = \nu^{}_{U_1} \smallboxplus  \nu^{}_{U_2}$ with finite signed measures $\nu^{}_{U_1}$ and $\nu^{}_{U_2}$ on $X_{U_1}$ and $X_{U_2}$, respectively. By the tensor product representation of $F$ from \eqref{tensor}, we have
\begin{equation*}
\begin{split}
F (\mu \smallboxtimes \nu^{}_{U_1}  & +  \mu \smallboxtimes \nu^{}_{U_2})  = F (\mu \smallboxtimes \nu^{}_{U_1}) + F (\mu \smallboxtimes \nu^{}_{U_2}) =  F (\pi^{}_{\owl{U_1}} . \mu  \otimes \nu^{}_{U_1}) +  F (\pi_{\owl {U_2}}^{} . \mu \otimes \nu^{}_{U_2})  \\
&=  ( \proj_{i_\ast} \otimes \id_{\owl {U_1} \setminus i_\ast}  ) (\pi^{}_{\owl {U_1}} . \mu ) \otimes \id_{U_1} (\nu^{}_{U_1}) +
 ( \proj_{i_\ast} \otimes \id_{\owl {U_2} \setminus i_\ast}  )(\pi^{}_{\owl {U_2}} . \mu ) \otimes \id_{U_2} (\nu^{}_{U_2})  \\
&= \pi^{}_{\owl {U_1}} .  ( \proj_{i_\ast} \otimes \id_{S^*} )(\mu)  \otimes \id_{U_1} (\nu^{}_{U_1}) 
+  \pi^{}_{\owl {U_2}} . \left ( \proj_{i_\ast} \otimes \id_{S^*} \right )(\mu)  \otimes \id_{U_2} (\nu^{}_{U_2}) \\
&= F(\mu) \smallboxtimes \nu^{}_{U_1} + F(\mu) \smallboxtimes \nu^{}_{U_2}, 
\end{split}
\end{equation*}
which gives the first claim.
Taking the first claim together with the fact that $f(\mu \smallboxtimes \nu) = f(\mu)$ if $\nu(X_{S^*}) = 1$, we get the second and the third claim.
\end{proof}

Now, the postponed proof becomes straightforward.
\begin{proof}[Proof of Theorem \textnormal{\ref{rekursion}}]
Let $\Psi^{(k)}$ be  as  in Definition \ref{hierarchy}. With the shorthand
\begin{equation*}
\nu_t^{(k-1)} \defeq \pi^{}_{D^{(k)}} . \int_0^t \varrho^{(k)} \ee^{-\varrho^{(k)} \tau}  \omega_\tau^{(k-1)} \dd \tau,
\end{equation*}
one has $\nu_t^{(k-1)}(X_{D^{(k)}})=1-\ee^{-\varrho^{(k)}t}$, and the right-hand side of the recursion formula from Theorem \ref{rekursion} can be expressed  as
\begin{equation}\label{rhs}
\mu_t^{(k)} \defeq \omega_t^{(k-1)} \smallboxtimes \, (\ee^{-\varrho^{(k)} t} \1 \smallboxplus \nu_t^{(k-1)}).
\end{equation}
First, we show that
\begin{equation}\label{sadkldht}
\mu_t^{(k)} \smallboxtimes \pi_{D^{(\ell)}}^{} . \mu_t^{(k)} = \big(\omega_t^{(k-1)} \smallboxtimes \ts \pi_{D^{(\ell)}}^{} .\omega_t^{(k-1)} \big) \smallboxtimes (\ee^{-\varrho^{(k)} t} \1 \smallboxplus \nu_t^{(k-1)})
\end{equation}
for all $1 \leq \ell \leq k$. To see this, write the left-hand side  as $\omega_t^{(k-1)} \smallboxtimes \kappa \smallboxtimes \chi$,
where
\begin{equation*}
\kappa \defeq \ee^{-\varrho^{(k)} t} \1 \smallboxplus \nu_t^{(k-1)}
\quad \text{and} \quad
\chi \defeq \pi_{D^{(\ell)}}^{}. \big (  \omega_t^{(k-1)} \smallboxtimes (\ee^{-\varrho^{(k)} t} \1 \smallboxplus \nu_t^{(k-1)})  \big ) = \pi_{D^{(\ell)}}^{} . \mu_t^{(k)}.
\end{equation*}
Recall that, by our monotonicity assumption on the permutation of sites, we have either $D^{(k)} \subseteq D^{(\ell)}$ or $D^{(k)} \cap D^{(\ell)} = \varnothing$. In the first case, \eqref{sadkldht} follows by cancelling $\kappa$ using Proposition~\ref{algebraprops} (note that $\kappa(X^{}_{D^{(k)}})=1$). In the second case, $\chi$ is just $\pi_{D^{(\ell)}}^{}. \omega_t^{(k-1)}$, so $\kappa \smallboxtimes \chi = \chi \smallboxtimes \kappa$, again by Proposition~\ref{algebraprops}. 
Now we compute, using \eqref{psireconew} and \eqref{rhs} in the first step, \eqref{sadkldht} and Lemma~\ref{equivariance} in the second,  Definition~\ref{hierarchy} in the third, and Proposition~\ref{algebraprops} in the last:
\begin{equation*}
\begin{split}
\Psi^{(k)}(\mu_t^{(k)}) & = \Psisel(\omega_t^{(k-1)} \smallboxtimes \big (\ee^{-\varrho^{(k)} t} \1 \smallboxplus \nu_t^{(k-1)}) \big ) + \sum_{\ell = 1}^{k} \varrho^{(\ell)} \mu_t^{(k)} \smallboxtimes  \, (\pi_{D^{(\ell)}}^{} . \mu_t^{(k)} \smallboxminus  \1  ) \\
 & = \big  ( \Psisel(\omega_t^{(k-1)}) + \sum_{\ell = 1}^{k} \varrho^{(\ell)} \omega_t^{(k-1)} \smallboxtimes \, (\pi_{D^{(\ell)}}^{} . \omega_t^{(k-1)} \smallboxminus \1) \big ) \smallboxtimes \, (\ee^{-\varrho^{(k)} t} \1 \smallboxplus \nu_t^{(k-1)}) \\
& = \big (\Psi^{(k-1)}(\omega_t^{(k-1)}) \smallboxplus  \varrho^{(k)} \omega_t^{(k-1)} \smallboxtimes \, (\pi_{D^{(k)}}^{} . \omega_t^{(k-1)} \smallboxminus \1) \big )  \smallboxtimes \, (\ee^{-\varrho^{(k)} t} \1 \smallboxplus \nu_t^{(k-1)}) \\[2mm]
& =  \dot \omega_t^{(k-1)}  \smallboxtimes \, (\ee^{-\varrho^{(k)} t} \1 \smallboxplus \nu_t^{(k-1)})  +  \omega_t^{(k-1)}\smallboxtimes \, ( \varrho^{(k)} \ee^{-\varrho^{(k)} t} \pi_{D^{(k)}}^{} .  \omega_t^{(k-1)} \smallboxminus \varrho^{(k)} \ee^{-\varrho^{(k)} t} \1). 
\end{split}
\end{equation*}
Identifying $ \varrho^{(k)} \ee^{-\varrho^{(k)} t} \pi_{D^{(k)}}^{} .  \omega_t^{(k-1)}$ with $\dot{\nu}_t^{(k-1)}$, we see that the last line is just the time derivative of $\mu_t^{(k)}$ of \eqref{rhs}.
\end{proof}

\begin{remark}
We could have proved Theorem \ref{rekursion} also without the help of formal sums and the new operations $\smallboxplus, \smallboxminus, \smallboxtimes$. However, we decided on the current presentation in order to familiarise the reader with this --- admittedly somewhat abstract --- formalism, as it is the key to stating the duality result in Section \ref{sec:duality} in closed form. It will also allow us later to state the solution itself in closed form; see Corollary \ref{explicitsolution}\hfill $\diamondsuit$
\end{remark}

\begin{remark}\label{generalisation}
Note that the only property of $\Psisel$ that entered the proof of Theorem \ref{rekursion} is the second property in Lemma~\ref{equivariance}. Therefore, the result remains true if $\Psisel$ is replaced by a more general operator with this property. In particular, Theorem \ref{rekursion} remains true when  frequency-dependent selection and/or mutation at the selected site is included. \hfill $\diamondsuit$
\end{remark}

An important application of Theorem~\ref{rekursion} is the following recursion for the first-order correlation functions $\omega_t^{(k)} - R^{(k)} \omega_t^{(k)}$ between the type frequencies at the sites contained in $C^{(k)}$ and those contained in $D^{(k)}$, for solutions of the truncated equations. These objects are referred to as \emph{linkage disequilibria} in  biology and also of independent interest \cite[Ch.~3.3]{durrett}.

\begin{lemma}[correlation functions]\label{linkage_disequilibria}
The family of solutions $ (\omega^{(k)})_{0 \leq k \leq n-1}$  of Definition~\textnormal{\ref{hierarchy}} satisfies, for    $1 \leq k \leq n-1$, 
\begin{equation*}\label{linde}
(\id - R^{(k)}) \omega_t^{(k)} = \mathrm{e}^{-\varrho^{(k)} t} ( \id - R^{(k)} ) \omega_t^{(k-1)}.
\end{equation*}
\end{lemma}

\begin{proof}
A direct verification via Theorem~\ref{rekursion}, using  $R^{(k)} \omega_t^{(k)} = \omega_t^{(k)} \smallboxtimes \pi^{}_{D^{(k)}} . \omega_t^{(k)}$.
\end{proof}

\section{Looking back in time: the ancestral selection-recombination graph}
\label{sec:asrg}
Our next goal is to reveal the genealogical content of the recursive solution  of Theorem \ref{rekursion}. We will accomplish this by a change of perspective: Instead of focusing on the evolution of the type distribution (of the entire population) forward in time as described by the SRE \eqref{main}, we will trace a single individual's  genealogy  \emph{back in time}.

The crucial tool  for this purpose is the \emph{ancestral selection-recombination graph} ($\ASRG$) of \cite{DonnellyKurtz,LessardKermany,BossertPfaffelhuber}. As the name suggests, it combines the \emph{ancestral selection graph} (ASG) of \cite{KundN} and the \emph{ancestral recombination graph} (ARG) of \cite{hudson,griffithsmarjoram96,griffithsmarjoram97}. We will introduce the $\ASRG$ here as taylored for the SRE; it allows to trace back, in a Markovian way, all lines that may carry information about the type (and the ancestry) of an individual at present. This  is similar to \cite{Cordero,BaakeCorderoHummel} for the selection part and to \cite{reco,haldane} for the recombination part, where the ancestral graphs consist of all \emph{potentially ancestral lines} of an individual at present. At this point, we  understand the notion of \emph{potentially ancestral} in a broad sense, indicating lines that are potentially ancestral to  some line in the graph, but not necessarily  to the individual at present. Indeed, some of these lines are not potentially ancestral to the present individual itself (that is, the notion of potential ancestry is not transitive); they will be pruned away later on. Consider first a finite population of size $N$. Recalling the definition of the Moran IPS in Section \ref{sec:Moran}, we can sample from the type distribution at present time $t$ via the following procedure (see Figure~\ref{sampling}).

\begin{figure}

\includegraphics[width = 0.83\textwidth]{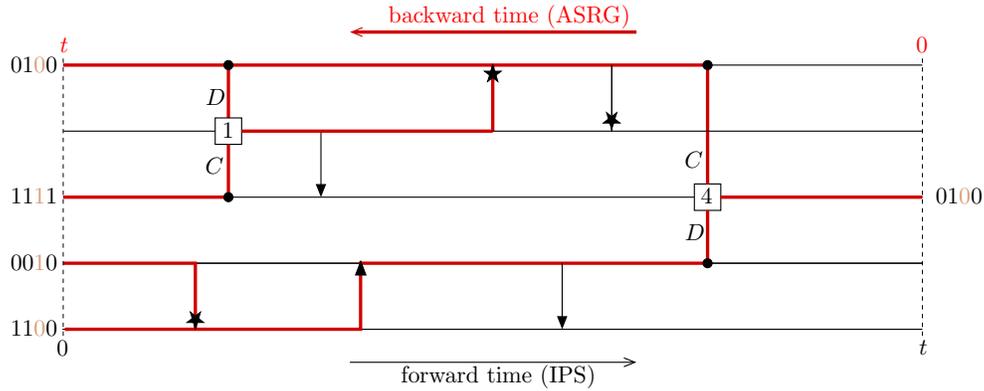}

\caption{\label{sampling} Sampling from the type distribution at present using the graphical representation of the Moran IPS. The $\ASRG$ is marked in red and the selected site in light brown. Notice the two different time axes for the IPS and the $\ASRG$, respectively; while the types are propagated through the IPS from left to right, the genealogy is constructed in the opposite direction, starting with a present-day individual on the right.}
\end{figure}

\begin{enumerate}
\item
\label{selectlabel}
Select an arbitrary label $\alpha$ from $\{1,\ldots,N\}$ for the individual to be considered.
\item
Construct the untyped version of the Moran IPS.
\item
\label{traceback}
Start the graph by tracing back the single line emerging from the individual at time $t$. Proceed as follows in an iterative way in the backward direction of time until the initial time is reached; note that forward time 0 (forward time $t$) corresponds to backward time $t$ (backward time 0).
\begin{enumerate}
\item
\label{relocation}
If a line currently in the graph is hit by the tip  of a neutral arrow, it is relocated to the line at the tail. 
\item 
\label{branching}
 If a line  in the graph is hit by a selective arrow, we trace back both its  potential ancestors, namely the incoming branch (at the tail of the arrow) and the continuing branch (at the tip). That is, we add the incoming line to the graph, which results in a \emph{branching event}.
\item
\label{splitting}
If a line is hit by a recombination square at site $i$, we have a \emph{splitting event} and trace back the lines that contribute  the head ($C_i$) and the tail ($D_i$), respectively, while the line hit by the square is discontinued.
\end{enumerate}
\item
\label{typing}
Assign types to all lines in the graph at time 0 by sampling without replacement from the initial counting measure $N Z^{(N)}_0$. Then, propagate the types  forward along the lines obtained in step (3), according to the same rules as in the Moran IPS. That is, selective branchings are resolved by applying the \emph{pecking order} derived from the Moran IPS and illustrated in Fig.~\ref{peckingorder}, namely: the incoming branch is parental to the descendant line if it has a $0$ at the selected site; otherwise, the continuing branch is parental. Splitting  events are resolved by piecing together heads and tails. This way, a type is associated with every line element of the graph.
\end{enumerate}

\begin{figure}[t] 
\includegraphics[width = 0.84\textwidth]{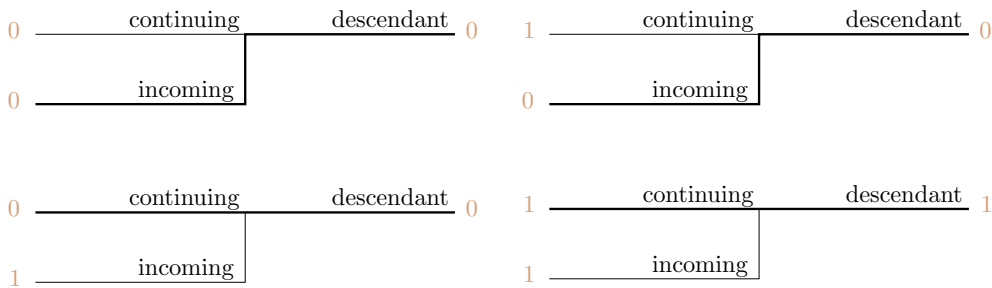}
\caption{\label{peckingorder} The pecking order between incoming line  and continuing line, and the resulting type of the descendant. In each case, the ancestral line is  bold. To keep the picture simple, we have only indicated the letter at the selected site. Likewise, the picture applies to the case $n=i_*=1$.}
\end{figure}

The graph resulting from steps (\ref{selectlabel})--(\ref{traceback}), along with the graphical elements indicating reproduction and recombination, is called the \emph{untyped $\ASRG$}, whereas the outcome of step (\ref{typing}) is the \emph{typed $\ASRG$}. While steps (\ref{relocation}) and (\ref{splitting}) are obvious, let us comment on the crucial branching step (\ref{branching}). It   reflects the fact that whether the incoming or the continuing branch  is the true parent depends on the type of the incoming branch, which is not known in the untyped situation; in this sense, every branching event encodes a case distinction. Let us also mention that, in all events (\ref{relocation})--(\ref{splitting}), it may happen that a  line coalesces with a line that is already in the graph. Likewise, it is possible that, in a splitting event, the same parent contributes both the head and the tail; the event then turns into a relocation.

Steps (\ref{selectlabel})--(\ref{typing}) yield the type of the present individual considered, but also serve to elucidate the true ancestry of each site in this individual. 
In step  (\ref{typing}), the paths along which  the individuals contributing to the type of the present-day individual are propagated are called \emph{(true) ancestral lines}, as opposed to the \emph{potentially ancestral lines} in the untyped $\ASRG$. More precisely, for $i \in S$, the path along which the type of the ancestor of site $i$ is propagated is called the \emph{ancestral line of site $i$}. It is obtained explicitly by adding step
\begin{enumerate}
\item[(5)] Trace back the ancestry of site $i$ by starting from the individual at present, following back the true ancestral line (determined in step (4)) in every branching event. This is the bold line in Fig.~\ref{peckingorder}, and the one following either the $C$ or $D$ branch at every splitting event, depending on whether $i \in C$ or $i \in D$. That is, we remove from the $\ASRG$ those lines that do not contribute genetic material to site $i$ in the present individual.
\end{enumerate}

Clearly, in step (2), we need not construct the full graphical representation of the interacting particle system. Instead, it suffices to consider those events that occur on the  lines in the $\ASRG$ of the sampled individual, that is, the lines (to be) traced back in step (3). We therefore obtain the same $\ASRG$ (in distribution) if steps (2) and (3) are replaced by the following single one.
\begin{enumerate}
\item[(2'\&3')]
Starting from the single line at forward time $t$,  move backward  and independently at rates $1,s$, and $\varrho_i$,  let each line in the graph be hit by neutral arrows, selective arrows, and recombination events at site $i$, $i \in S^*$, with the (potential) parent individual(s)  chosen uniformly  without replacement from $[1:N]$; update the graph accordingly. 
\end{enumerate}
Note that we use  the   homogeneity of the Poisson process here, which entails that  the graphical elements are laid down according to the same law in either direction of time. Note also that the probability of choosing, for any kind of event,  parent(s) already contained in the genealogy is of order $1 / N$; the same is true for the probability to choose the same parent twice in a recombination event. In the limit $N \to \infty$, therefore, the coalescence rate vanishes. Likewise,  selective reproduction (recombination) events  always result in branching (splitting),  with the incoming branch (both arms) outside the current set of lines. Furthermore, we disregard the position of the lines within the IPS; this is allowed because the types associated with each line form a \emph{permutation-invariant} or  \emph{exchangeable} family of random variables. In particular, therefore, relocations may safely be ignored. The resulting random graph is called the $\ASRG$ \emph{in the LLN regime}. Since we will only be concerned with this limit in the remainder of the paper, we will often omit this specification.
\begin{definition}\label{defASRG}
For any given $t > 0$, the \emph{ancestral selection-recombination graph} ($\ASRG$) \emph{in the \textnormal{LLN} regime} is a random graph-valued function in backward time starting from a single node at time $0$ and growing from right to left until time $t$, where 
branching events 
\begin{figure}[H]
\psfrag{dots}{\ldots}
\includegraphics[width = 85mm]{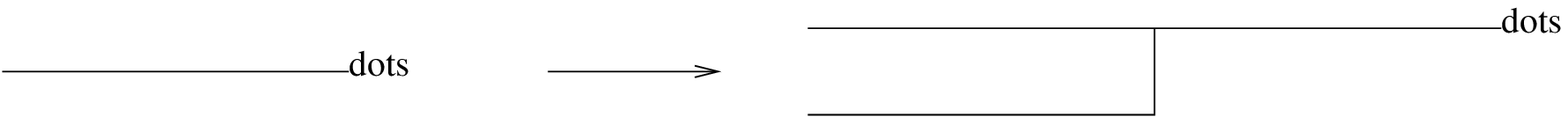}
\end{figure}
\noindent occur at rate $s$ on every line, and splitting events 
\begin{figure}[H]
\psfrag{i}{$i$}
\psfrag{dots}{\ldots}
\includegraphics[width = 85mm]{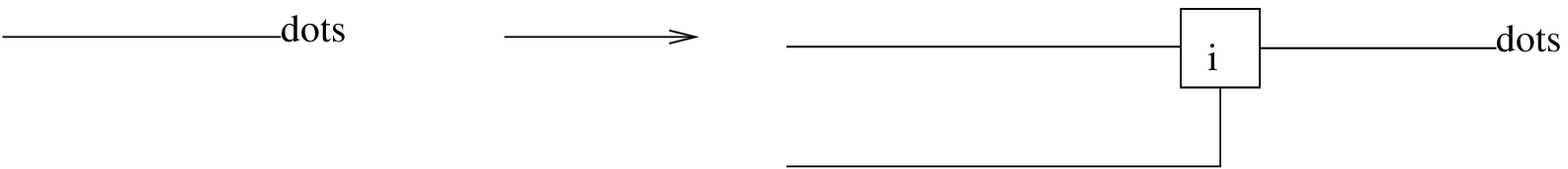}
\end{figure}
\noindent occur at rate $\varrho_i$, $i \in S^*$, per line; all events are mutually independent. The rightmost node is called the \emph{root} of the $\ASRG$ and the leftmost nodes are called the \emph{leaves}. 
\end{definition}

The ASRG is almost surely finite, that is, an ASRG of finite length contains only a finite number of branchings / splittings. Note  that  we dispense with the star-shaped arrowheads used in the IPS for the selective events; rather, we use the convention that the incoming branch be placed below the continuing branch. This is again allowed due to exchangeability. For the same reason, we dispense with the labelling of the recombination arms  and instead adopt the convention that the sites in the head always come from the individual on the upper line, which we place on the same level as the descendant line. The sites in the tail are provided by the line attached from below. 
For an example realisation of the $\ASRG$ and the construction of the type of an individual at present along with the ancestral line of one specific site, see Fig.~\ref{ASRG}. 

\begin{figure}
\includegraphics[width = 0.75\textwidth]{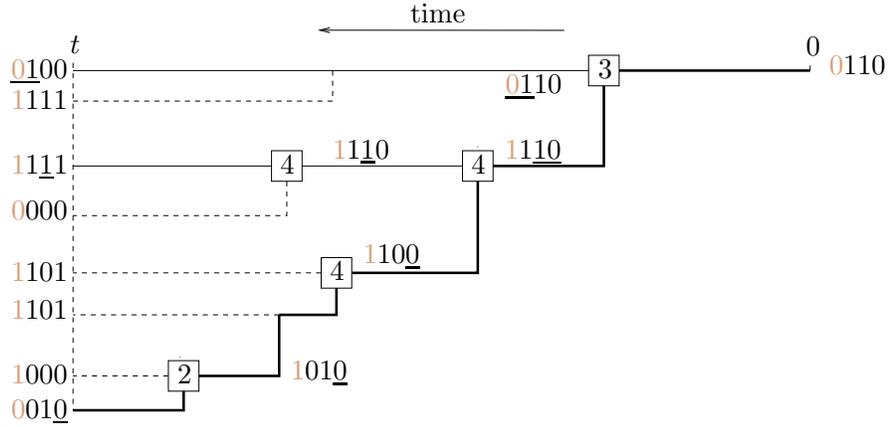}
\caption{\label{ASRG} Tracing back the ancestry of an individual with 4 sites $i_0 = 1$, \mbox{$i_1 = 2$}, $i_2 = 3$ and $i_3 = 4$ under selection and recombination; the selected site $i^{}_*=1$ is light brown. The bold line is ancestral to site $4$, the thin solid lines are ancestral  to sites 1, 2, or 3, and the dashed lines are not ancestral to any site.  Each branch is decorated with its type, and the sites to which it is ancestral are underlined.}

\end{figure}

The $\ASRG$  implies the following sampling procedure for $\omega^{}_t$. First, construct a realisation of the $\ASRG$, run for time $t$. Then, assign types to its leaves, sampled independently from $\omega_0$, and propagate them through the graph as described above. 

\begin{remark}
In order to  connect the graphical constructions in this section to the  viewpoint from the previous section, let us describe the type propagation in slightly more formal terms. Given a realisation of the $\ASRG$ of length $t$, we assign to each node a type distribution as follows. First, each leaf is assigned the initial type distribution $\omega_0^{}$. If an internal node $v$ arises due to a branching, we associate to $v$ the distribution
$\omega_v := f(\omega_{\text{inc}}) b(\omega_{\text{inc}}^{}) + \big (1 - f(\omega_{\text{inc}}) \big) \omega_{\text{cont}}^{}$,
that is,
\psfrag{dots}{\ldots}
\psfrag{omegav}{$f(\omega_{\text{inc}}) b(\omega_{\text{inc}}^{}) + \big (1 - f(\omega_{\text{inc}}) \big) \omega_{\text{cont}}^{}$}
\psfrag{omegacont}{$\omega_{\text{cont}}$}
\psfrag{omegainc}{$\omega_{\text{inc}}$}
\hspace*{-2cm}
\begin{center}
\includegraphics[width = 0.5\textwidth]{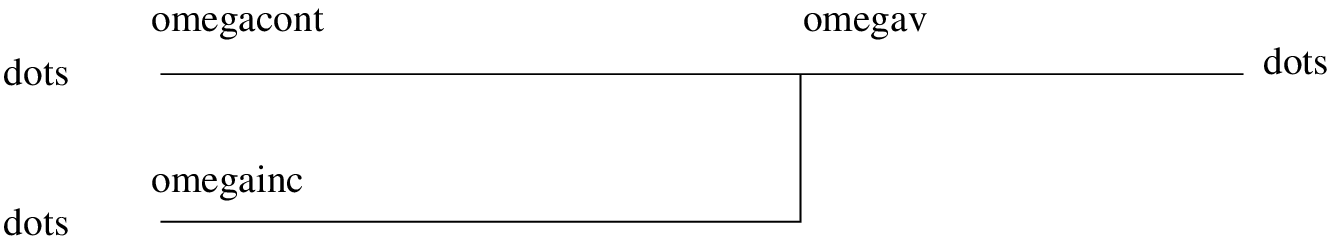}
\end{center}
where  $\omega_{\text{inc}}^{}$ and $\omega_{\text{cont}}^{}$ are the type distributions associated to the nodes that connect to $v$ via the incoming and continuing branch.

Likewise, if $v$ is due to splitting (at site $i$, say), we associate with it  
$\omega_{\text{head}}^{\phantom{D_i}} \smallboxtimes \omega_{\text{tail}}^{D_i}$,
where $\omega_{\text{head}}^{}$ and  $\omega_{\text{tail}}^{}$ are the distributions associated to the nodes that connect to $v$ via the ancestral lines of the head and tail, respectively, \\
\psfrag{dots}{\ldots}
\psfrag{omegav}{$\omega_{\text{head}} \smallboxtimes \omega_{\text{tail}}^{D_i}$}
\psfrag{omegahead}{$\omega_{\text{head}}$}
\psfrag{omegatail}{$\omega_{\text{tail}}$}
\begin{center}
\psfrag{i}{$i$}
\includegraphics[width = 0.5\textwidth]{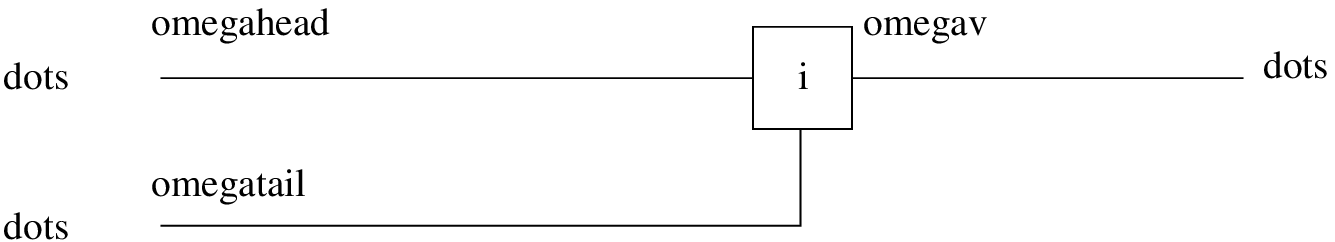}
\end{center}
Finally, the distribution for the root  equals  that of the unique internal node connected to it. \hfill $\diamondsuit$
\end{remark}

\begin{example} \label{example_puresel}
In the case of pure selection ($k = 0$), our $\ASRG$ reduces to an ordered version of the ASG in the deterministic limit;  this is equivalent to a special case of the \emph{pruned lookdown ASG} in the LLN regime, as introduced in \cite{Cordero,BaakeCorderoHummel} in the context of a probabilistic representation of the solution of the deterministic selection-mutation equation. Since there are no coalescence events  in this regime, the number of lines in the graph, that is, the number of potential ancestors of an individual sampled at time $t$, is a simple Yule process $K=(K_t)_{t \geq 0}$ with branching rate $s$.  This is a continuous-time branching process where, at any time $t$, every individual branches into two at rate $s$, independently of all others.  In the case considered here, the process starts with $K_0=1$.  Clearly, the pecking order implies that the individual at present will be drawn from the unfit subpopulation $d(\omega^{}_0)$ if  all $K_t$ potential ancestors are of deleterious type, which happens with probability $( 1 - f(\omega_0)  )^{K_t}$.  Otherwise (with probability $1- \big ( 1 - f(\omega_0) \big  )^{K_t}$),  the individual will be sampled from the fit subpopulation $b(\omega^{}_0)$. Thus, we obtain a \emph{stochastic representation} of the solution of the selection equation by averaging over  all realisations of the Yule process at time $t$:
\begin{equation}\label{omega_selection_only}
\begin{split}
\omega_t^{(0)} & = \varphi^{}_t(\omega^{}_0) \\
& = \EE \big [ \big ( 1 - f(\omega^{}_0) \big )^{K_t} \mid K_0 = 1 \big ] d(\omega^{}_0) +  \big (1 - \EE \big [ \big ( 1 - f(\omega^{}_0) \big )^{K_t} \mid K_0 = 1 \big ] \big) b(\omega^{}_0).
\end{split}
\end{equation}

It is well known that  $K_t$, given $K_0=1$, follows $\geom(\ee^{st})$ (cf.~\cite[Ch.~II.4]{Feller_68} or \cite[Ex.~2.19]{SO94}), where $\geom(\sigma)$ denotes the distribution of the number of independent Bernoulli trials with  success probability  $\sigma$ up to and including the first success.
The probability generating function is given by
\begin{equation}\label{genfunc}
g(z) = \EE \big [z^{K_t} \mid K_0=1 \big ] = 
\myfrac{\mathrm{e}^{-st} z}{1 - (1 - \mathrm{e}^{-st})z}.
\end{equation}
Consequently, 
\begin{equation}\label{unfitviayule}
  \EE \big [ \big (  1 - f(\omega^{}_0) \big )^{K_t} \mid K_0 = 1 \big ] = \myfrac{\ee^{-st}(  1 - f(\omega^{}_0) \big ) }{\ee^{-st}(  1 - f(\omega^{}_0) \big ) + f(\omega^{}_0) } = 1 - f(\omega^{(0)}_t)
\end{equation}
with $f(\omega^{(0)}_t)$ of Proposition~\ref{omeganull}.
Inserting this into \eqref{omega_selection_only}, we obtain $\omega_t^{(0)}$ of Proposition~\ref{omeganull}.

Anticipating the results in Section 6, this can be viewed as a special case of the general duality relation with respect to the duality function 
\begin{equation}\label{defH}
h(m,\mu) = \big ( 1 - f(\mu) \big )^m_{} d(\mu) + \big (1 - \big (1 - f(\mu) \big )^m_{} \big ) b(\mu)
\end{equation}
(cf.\ Definition \ref{duality} and Proposition \ref{omeganullduality}), which is the distribution of an individual's type at present, given it has $m$ potential ancestors sampled from $\mu \in \cP(X)$.  \hfill $\diamondsuit$
\end{example}

\begin{example}\label{example_purereco}
Likewise, in the case of pure recombination, the $\ASRG$ reduces to a stochastic partitioning process $\varSigma = \left ( \varSigma_t \right )_{t \geq 0}$; this is a special case of   \cite[Sec.~6]{reco} or \cite{haldane}, where recombination is tackled as a more general partitioning process, rather than the single-crossover case treated here. In our case,
$\varSigma$ is a continuous-time Markov chain on the lattice of interval partitions of $S$ whose law is simply stated as follows. Start with $\varSigma_0 = \{S\}$ and,  if the current state is $\varSigma_t$, a transition to state
$\varSigma'_t \defeq \varSigma_t \wedge \{C_i,D_i\}$ occurs
at rate $\varrho_i$ for $i \in S^*$. Here, $\cA \wedge \cB$ denotes the coarsest
common refinement of  partitions $\cA$ and $\cB$, that is,
$\cA \wedge \cB \defeq \{A \cap B : A \in \cA, B \in \cB \} \setminus \{\varnothing\}$.
Note that this includes silent events, where $\varSigma'_t=\varSigma^{}_t$.
Given $\varSigma_t$, one can sample an individual of type $x = (x_1,\ldots,x_n)$ from the distribution $\omega_t$ as follows. First, construct a realisation $\sigma = \{A_1,\ldots,A_k\}$ of $\varSigma_t$. Then, sample individuals $\cX(1),\ldots,\cX(k)$ i.i.d.\ from the initial type distribution $\omega^{}_0$ and set  
\begin{equation*}
x \defeq \big ( \pi^{}_{A_1} \big (\cX(1) \big ),\ldots, \pi^{}_{A_k} \big ( \cX(k) \big ) \big ),
\end{equation*}
which has distribution $\tilde R_{\sigma} (\omega_0)$; see Eq.~\eqref{generalrecombinator}.
Averaging over all realisations of $\varSigma_t$ gives
\begin{equation}\label{Beispiel fuer Dualitaet}
\omega_t = \EE \big [ \tilde R_{\varSigma_t} (\omega_0) \mid \varSigma_0 = \{S\} \big ].
\end{equation}
As in the previous example, this can again be interpreted as a special case of a duality relation, this time with respect to the duality function 
$\tilde{H}(\cA,\mu) = \tilde R_{\cA} (\mu)$,
see Definition \ref{duality}. \hfill $\diamondsuit$
\end{example}

We now turn to the genealogical proof of  Theorem \ref{rekursion}; recall that the start of the recursion, the solution $\omega^{(0)}$ of the pure selection equation, was already considered in Example \ref{example_puresel}. We reuse the permutation $(i_k)^{}_{0 \leq k \leq n-1}$ of sites defined in Section \ref{sec:solution} and, in perfect analogy with the family $(\omega^{(k)})_{0 \leq k \leq n-1}$,  define for $0 \leq k \leq n-1$  the \emph{$\ASRG$ truncated at $k$} to be an  $\ASRG$ with $\varrho^{(\ell)}=0$ for all $\ell > k$. We denote the  $\ASRG$ truncated at $k$ by   $\ASRG^{(k)}$, or by  $\ASRG^{(k)}_t$ if we want to indicate its duration. Clearly, the $\ASRG^{(k)}$ is the $\ASRG$ corresponding to $ \omega^{(k)}$. In particular, the $\ASRG^{(0)}$ is just the ASG (without recombination), and  the type at the root of an  $\ASRG^{(k)}_t$ follows $\omega^{(k)}_t$. The key ingredient to the genealogical proof of the recursion  is the following proposition, which links the type of the root of an $\ASRG^{(k)}$ to the type at the root of an $\ASRG^{(k-1)}$, or two independent copies thereof.

\begin{prop}\label{root_decomp_refined}
For $1 \leq k \leq n-1$ and any given $t > 0$, let  $B$ be a Bernoulli variable with success probability  $1-\ee^{-\varrho^{(k)} t}$. Conditional on $\{B=1\}$, let $T$  be an $\exponential(\varrho^{(k)})$ random variable conditioned on  being $\leqslant t$, where $\exponential(\sigma)$ denotes the exponential distribution with parameter $\sigma$. Furthermore, denote by $\fX \in X$ the type at the root of an $\ASRG^{(k-1)}_{t}$, and by $\widetilde \fX$ the type at the root of an $\ASRG^{(k-1)}_{T}$, independent of the $\ASRG^{(k-1)}_t$ that delivers $\fX$. The type 
$\fZ$  at the root of an $\ASRG^{(k)}_t$ is then, in distribution, given by
\[
\fZ = (1-B) \fX + B  \big ( \pi_{C^{(k)}}(\fX),\pi_{D^{(k)}}(\widetilde \fX) \big ).
\]
\end{prop}

Before we prove this, let us give some intuition. We work with the untyped $\ASRG^{(k)}_t$, obtained via steps (1) and (2'\&3'), and consider the  line ancestral to $D^{(k)}$. It is clear that this is a single line because, due to the partial order,  none of the splitting events in the $\ASRG^{(k)}$ partition $D^{(k)}$. Note that the location of the true  ancestral line is not yet known, since this is only decided in step (4), when propagating the types forward, as in Figure~\ref{ASRG}.  

We distinguish two cases. With probability $\ee^{-\varrho^{(k)} t}$, no splitting at site $i_k$ has occurred along this line, so the tail is `glued' to the head. Thus,  $\fZ$ may be constructed as in the absence of recombination  at site $i_k$, that is, via an $\ASRG^{(k-1)}_{t}$; this gives the first term on the right-hand side.  With probability $1-\ee^{-\varrho^{(k)} t}$, a splitting at site $i_k$ has occurred along the ancestral line of $D^{(k)}$. We then consider the time of the last, that is, of the \emph{leftmost} splitting event at site $i^{}_k$ on the line in question and identify this time with $t-T$ (since such splitting events occur at rate $\varrho^{(k)}$ and due to the homogeneity of the Poisson process, $T$ is indeed distributed as stated). The ancestry of the sites  in $C^{(k)}$ is then  unaffected by the split and thus follows an  $\ASRG^{(k-1)}_{t}$ (in line with marginalisation consistency, see Theorem~\ref{consistency}). But the sites contained in $D^{(k)}$ now come from a different  individual. Since $t-T$ is the time of the \emph{leftmost} splitting event, we know that no further splits at site $i_k$ have occured at any point further back in the past. This means that, at this point, the tail of the individual at the root of an independent $\ASRG^{(k-1)}_{T}$ enters the ancestral line. The combination of head and tail as described gives the second term on the right-hand side. 

In order to turn these heuristics into a proof, we have to make the construction of the ancestral line of $D^{(k)}$ explicit.
To this end, we mimick the recursion forward in time by coupling the $\ASRG_t^{(k)}$ to an $\ASRG_t^{(k-1)}$. To keep things transparent, we introduce the following simplified construction; see Fig.~ \ref{collapsedASRG}. 

\begin{figure}
\includegraphics[width = 0.8\textwidth]{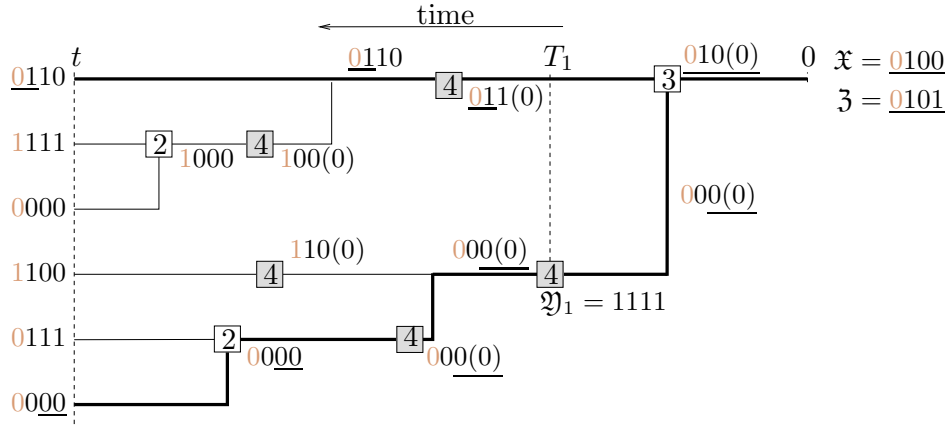}

\caption{
  \label{collapsedASRG}
  Determining the type at the root of  a cASRG$^{(4)}$. The graph is a cASRG$^{(4)}$, the selected site is light brown, ancestral lines in the $\ASRG^{(3)}$ are printed in bold, and ancestral letters are underlined. The shaded recombination squares indicate splitting events at site 4, where a new copy of an $\ASRG^{(4)}$ is attached for the remaining time.  Parentheses mark the  4th site in the $\ASRG^{(3)}$ that is replaced by the tail of the new copy. Thus, $\fX$ is obtained by ignoring the shaded squares as well as the parentheses, and $\fZ$ is then obtained by replacing the  $0$ in brackets in the type of the lower branch of the rightmost recombination event by the $1$ from $\fY_1$. 
}

\end{figure}

\begin{definition}[collapsed ASRG] \label{cASRG}
Let $1 \leq k \leq n-1$ be given. A \emph{collapsed $\ASRG$ truncated at $k$}, or $\text{cASRG}^{(k)}$, is an $\ASRG^{(k-1)}$ decorated with $i_k$-recombination squares  laid down according to independent Poisson point processes at rate $\varrho^{(k)}$  on every horizontal line segment. 
\end{definition}

We can then construct a realisation of the $\ASRG_t^{(k)}$ by attaching to every $i_k$-recombination square of a $\text{cASRG}^{(k)}$ an independent copy of an $\ASRG^{(k)}$ for the remaining time; that is, for any  $i_k$-recombination square at time $\tau \in [0,t]$, we attach an $\ASRG_{t - \tau}^{(k)}$. So splitting events take the form of attachment events. In the subsequent sampling step, this attachment provides the $k$-tail, while the $k$-head comes from the original $\ASRG_t^{(k-1)}$.  
We now describe how to use the collapsed $\ASRG$ to sample a root individual of an $\ASRG^{(k)}_t$, that is, to sample from $\omega_t^{(k)}$. First, one constructs a realisation of the $\text{cASRG}^{(k)}_t$. Then, types are assigned to the leaves according to $\omega_0^{}$ in an i.i.d.~fashion and propagated forward, where selective branchings and splitting (attachment) events are resolved as in the $\ASRG$. Assume an $i_k$-square is encountered on a given line at some (forward) time $\tau\in [0,t]$, and the type just before the $i_k$-square (that is, at time $\tau-0$)  is $\fx$. We then draw a  type $\fy$  from $\omega_{\tau}^{(k)}$, independently of $\fx$, for the individual contributing the tail. The type on the line then jumps from $\fx$ at time $\tau-0$ to type  
$\fz = \big ( \pi_{C^{(k)}}^{} (\fx), \pi_{D^{(k)}}^{} (\fy) \big )$
at time $\tau$, see Fig.~\ref{cASRGrule}. 
Keeping in mind the original motivation behind Definition \ref{cASRG} and thinking of the $i_k$-squares as splitting events (at site $i_k$) at which a new realisation of an $\ASRG^{(k)}$ is attached, it is clear that this gives the correct result. 

\begin{figure}[b]
\includegraphics[width = 0.55\textwidth]{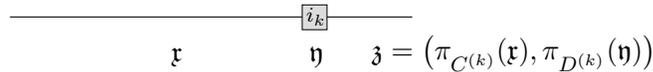}

\vspace{2mm}

\caption{
\label{cASRGrule}
Upon encountering an $i_k$-square,  the head of type $\fx$  is combined with the tail of a newly sampled individual (from $\omega_{\tau}^{(k)}$)  to form the type of the descendant. 
}
\end{figure}

\begin{proof}[Proof of Proposition~\textnormal{\ref{root_decomp_refined}}]
Let $1 \leq k < n$ and $t > 0$ be fixed and let a realisation of the $\text{cASRG}^{(k)}_t$ be given, together with an assignment of types to its leaves. Elements of the proof are illustrated in Fig.~\ref{collapsedASRG}.
Note first that
\begin{itemize}
\item
  $\fX$ is, in distribution, equal to the type of the root when ignoring the $i_k$-squares.
\end{itemize}
We consider the line ancestral to $D^{(k)}$ in the underlying $\ASRG_t^{(k-1)}$. The location of this line is now well defined, since we sample the types and can perform steps (4) and (5). Note that the line ancestral to $D^{(k)}$ is, at the same time, the line ancestral to $\overleftarrow{i_k^{}}$, the predecessor of $i_k^{}$; this is because no splits happen at $i_k$  in the $\ASRG^{(k-1)}_t$. We consider the following quantities.
\begin{itemize}
\item
  Let $B_1$ be the Bernoulli variable that takes the value 0 (the value 1) if there is no (at least one) recombination square on the ancestral line of $D^{(k)}$. Clearly, $B_1$  has success probability  $1-\ee^{-\varrho^{(k)} t}$.
\item
Conditional on $\{B_1=1\}$, let $T_1$ be the waiting time for the  first  $i_k$-square, in the backward direction of time, on the line ancestral to $D^{(k)}$ (that is, the \emph{rightmost}  $i_k$-square on this line in our graphical representation). Clearly, $T_1$ is  an $\exponential(\varrho^{(k)})$-random variable  conditioned to be $\leqslant t$, and independent of $\fX$.
\item
Let $\fY_1 \in X$ be the type of the root of the independent $\ASRG^{(k)}_{t-T_1}$ attached upon encountering the $i_k$-square at time $T_1$, that is, an  independent sample from $\omega_{t-T_1}^{(k)}$. 
\end{itemize}
We then have (cf. Fig.~\ref{collapsedASRG})
\begin{equation}\label{root_decomp}
\fZ = (1- B_1) \fX + B_1 \big ( \pi_{C^{(k)}}(\fX),\pi_{D^{(k)}}(\fY_1) \big ).
\end{equation}

\begin{figure}
\includegraphics[width = 0.5\textwidth]{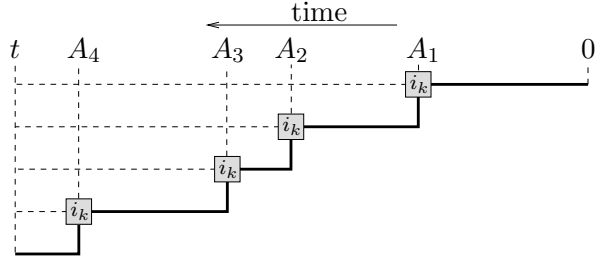}
\caption{\label{poissonset}
The ancestral line of $D^{(k)}$ after expanding all  recombination events arriving at  elements of  $W \cap [0,t]$ used in the  proof of Proposition~\ref{root_decomp_refined}. The ancestral lines of the corresponding heads (dashed)  need not be considered any further. Note that the maximal element $A_4$ is  the leftmost one.
}
\end{figure}

We now iterate Eq.~\eqref{root_decomp}, see Figure \ref{poissonset}.
In the first step, we draw  $\fX$ and $B_1$ as above. If $B_1 = 1$, we also draw $T_1$ according to $\exponential(\varrho^{(k)})$, conditioned on being $\leqslant t$.
If $B_1=0$, we set $\fZ = \fX$. If $B_1=1$,  by Eq.~\eqref{root_decomp} we must construct $\fY_1$, which contributes the tail. Since $\fY_1$ is the type at the root of an $\ASRG^{(k)}_{t-T_1}$,  
we do this by applying Eq.~\eqref{root_decomp} to $\fY_1$ instead of $\fZ$, that is, we repeat the first step but replace $t$ by $t-T_1$. So we determine whether or not there is a recombination square on the ancestral line between 0 and $t-T_1$; if there is one, we determine the waiting time for it, and so forth. More explicitly, let $B_2$ be the new indicator variable, which is Bernoulli with success probability $1-\ee^{-\varrho^{(k)}(t-T_1)}$. If $B_2=0$, let $\fX_1$ be the type at the root of an independent copy of the $\ASRG^{(k-1)}_{t-T_1}$. If $B_2=1$, let $T_2$ be the waiting time for the new event; $T_2$ follows $\exponential(\varrho^{(k)})$  conditioned to be $\leqslant t-T_1$;  and let $\fY_2$ be the type at the root of an  independent $\ASRG^{(k-1)}_{t-T_1-T_2}$. Inserting this back into Eq.~\eqref{root_decomp}, we obtain
\begin{equation*}
\fZ = (1-B_1) \fX + B_1(1-B_2) \big (\pi_{C^{(k)}}^{} (\fX), \pi_{D^{(k)}}^{} (\fX_1) \big ) + B_1 B_2 \big (\pi_{C^{(k)}}^{} (\fX), \pi_{D^{(k)}}^{} (\fY_2) \big );
\end{equation*}
note that, if $B_1=0$,  $B_2$ has not been declared, but the terms involving it remain well-defined since $B_1$ vanishes. 
Iterating this further gives
\begin{equation}\label{iteratefurther}
 \fZ = (1-B_1) \fX + \sum_{i\geq 1} B_1\cdot \ldots \cdot B_i (1-B_{i+1})  \big (\pi_{C^{(k)}}^{} (\fX), \pi_{D^{(k)}}^{} (\fX_i) \big ), 
\end{equation}
where  $\fX_i$ is the type at the root of  an independent $\ASRG^{(k-1)}_{t - \sum_{j=1}^i T_j}$, and we adhere to the above convention concerning undeclared $B_i$. Note that, with probability 1, exactly one of the terms on the right-hand side is nonzero; in particular, $B_1\cdot \ldots \cdot B_i=0$  whenever $\sum_{j=1}^i T_j>t$,  so everything is well defined. 

Let us now interpret the arrival times $T_j$ of the $i_k$-squares as  arrival times in a Poisson point set $W$ with intensity measure $\varrho^{(k)}\one_{t \geq 0} \dd t$   and elements $A_1 < A_2 < \ldots$.  When $A_i \leq t$, we have $A_i = \sum_{j=1}^i T_j$. Furthermore, 
$B_1 = \one_{\{A_1 \leqslant t\}}$ and, for $i \geq 1$,
\begin{equation}\label{poissonsum}
 B_1\cdot \ldots \cdot B_i  = \one_{\{A_i \leq t\}}, \qquad \text{as well as} \quad
B_1\cdot \ldots \cdot B_i (1- B_{i+1})  = \one_{\{A_i \leq  t < A_{i+1}\}}.
\end{equation}
We now rewrite $B_1$ as $B_1 = \one_{\{  W \cap [0,t] \neq \varnothing \}}$. Together with \eqref{poissonsum}, this entails  that the nonzero term   in \eqref{iteratefurther} is the first one if $W \cap [0,t]$ is empty; and if the set is nonempty, then the nonzero term is the one with the index $i$  that satisfies    $A_i = \max(W \cap [0,t])$. Conditionally on $B_1=1$, we therefore set $T \defeq t - \max(W \cap [0,t])$.
The claim then follows by identifying $B$ with $B_1$, and by noting that $T$ has the same distribution as $T_1$, namely $\exponential(\varrho^{(k)})$ conditioned to be $\leq t$.
\end{proof}

\begin{remark}
Remembering the motivation for the collapsed $\ASRG^{(k)}$, we think of every $i_k$-square as the anchor point for a new independent copy of the $\ASRG^{(k)}$, which is collapsed to keep things tidy. In  the above proof,  we iteratively expand the $i_k$-squares \emph{on the ancestral line of $D^{(k)}$} until there are no more recombination events left on that  line. Therefore, the Poisson point set $W$ has an interpretation as the collection of all recombination events  on the ancestral line of $D^{(k)}$. The proof has made precise the previously heuristic notion  of the last splitting event at site $i_k$ encountered on the ancestral line of $D^{(k)}$ in the backward direction of time; that is, the leftmost  event in the graphical representation, see Fig.~\ref{poissonset}. \hfill $\diamondsuit$
\end{remark}

\begin{remark}
When sampling $\fY_1$ via the newly attached $\ASRG^{(k)}$ in \eqref{root_decomp}, one might wonder whether it would suffice to construct  the potential ancestry of  the tail alone --- after all, the head of $\fY_1$ does not enter $\fZ$.
However, it cannot be overemphasised that this is not the case! Although $\fY_1$ only contributes the tail, the branching events in its ancestry can only be resolved if the letter at the selected site is known, whence we need to also trace back  the ancestry of the head attached to the new tail. 
We are haunted here by  the fact that marginalisation consistency does not hold for the tail, see the Appendix, in particular  Remark \ref{inconsistency}. \hfill $\diamondsuit$
\end{remark}

We are now all set to re-prove Theorem \ref{rekursion}. Indeed, Proposition~\ref{root_decomp_refined} connects  the random variable $\fZ$, delivered by an $\ASRG^{(k)}$, with random variables $\fX$ and $\widetilde \fX$ , delivered by an $\ASRG^{(k-1)}$. This is the crucial observation that we will now exploit.

\begin{proof}[Genealogical proof of Theorem \textnormal{\ref{rekursion}}]
From Proposition \ref{root_decomp_refined}, we can extract the conditional distribution of $\fZ$ given $B$ and $T$:
\begin{equation*}
\PP(\fZ = \bs{\cdot} \mid B,T) = (1-B) \ts \omega_t^{(k-1)} + B \pi_{C^{(k)}}^{} . \omega_t^{(k-1)} \otimes \pi_{D^{(k)}}^{} . \omega_T^{(k-1)}.
\end{equation*}
Theorem \ref{rekursion} now follows by integrating out $B$ and $T$, keeping in mind their distributions (and denoting the distribution of the latter by $\lambda$):
\begin{equation*}
\begin{split}
\omega_t^{(k)}& = \PP(B=0) \ts \omega_t^{(k-1)} + \PP(B=1)  \ts \pi_{C^{(k)}} . \omega_t^{(k-1)} \otimes \int_0^\infty \pi_{D^{(k)}} . \omega_\tau^{(k-1)} \dd \lambda(T) \\[1mm]
&= \ee^{-\varrho^{(k)} t} \omega_t^{(k-1)}  + \pi_{C^{(k)}} . \omega_t^{(k-1)} \otimes \int_0^t \varrho^{(k)} \ee^{-\varrho^{(k)} t} \pi_{D^{(k)}} . \omega_\tau^{(k-1)} \dd \tau,
\end{split}
\end{equation*}
and we are done.
\end{proof}

\section{Interlude}
\label{sec:interlude}
Using our insight from the proof of Proposition~\ref{root_decomp_refined}, we now informally describe  a more efficient version of the $\ASRG$ in order to motivate the more elegant dual process and the formal duality result that are detailed and proved in the next section. We start with an untyped $\ASG=\ASRG^{(0)}$, since this marks  the beginning of the recursion. Recall  that, in the  iteration leading from $\omega^{(0)}$ to $\omega^{(1)}$ via the $\cASRG^{(1)}$, $i_1$-recombination squares are laid down at rate $\varrho^{(1)}$ independently on every line of the ASG.  But at most one of these squares turns out as  relevant; namely the rightmost square on the ancestral line of $D^{(1)}$, if there is such a square.  Recall also that the head of the root  of the $\ASRG^{(1)}$, that is its sites in $C^{(1)}$, are delivered by the initial ASG, independently of any recombination squares; while the sites in the tail are delivered by an independent copy of the $\ASRG^{(1)}$, attached below the square for the remaining time and processed in the same way as the initial one. This procedure stops when no further recombination square is found on the ancestral line of the tail.  

In order to reduce the $\ASRG^{(1)}$ to its essential parts, we now start over and decorate the ASG with \emph{at most one recombination event}, which will play the role of the relevant one, see Figure~\ref{essential}. Namely, with probability $\ee^{-\varrho^{(1)}t}$, we include no event, and both head and tail are delivered by the ASG. With probability $1-\ee^{-\varrho^{(1)}t}$, we include one event, which happens at time $T_1$ distributed according to $\exponential(\varrho^{(1)})$ conditioned to be $\leq t$. Since we are in an untyped setting and do not know which of the lines in the ASG will be ancestral to the head, we symbolise the event by an $i_1$-\emph{box} (that is, a box labelled $i_1$) that covers all lines. At the bottom of the box, we attach an independent copy of the ASG starting with a single line and running for the remaining time. The new ASG is processed in the analogous way, with $t$ replaced by $t-T_1$. This procedure stops when no further $i_1$-box is encountered; this is (almost surely) the case after a finite number of steps, see  Figure~\ref{essential} (top left). The initial ASG delivers the head, while the last ASG attached delivers the tail. In particular, at every $i_1$-box, the tail delivered by the ASG attached below is combined with the head of whichever of the lines running through the box will turn out to be ancestral to the root of the ASG it belongs to.

\begin{figure}
 \begin{center}
  \includegraphics[width = 0.4\textwidth]{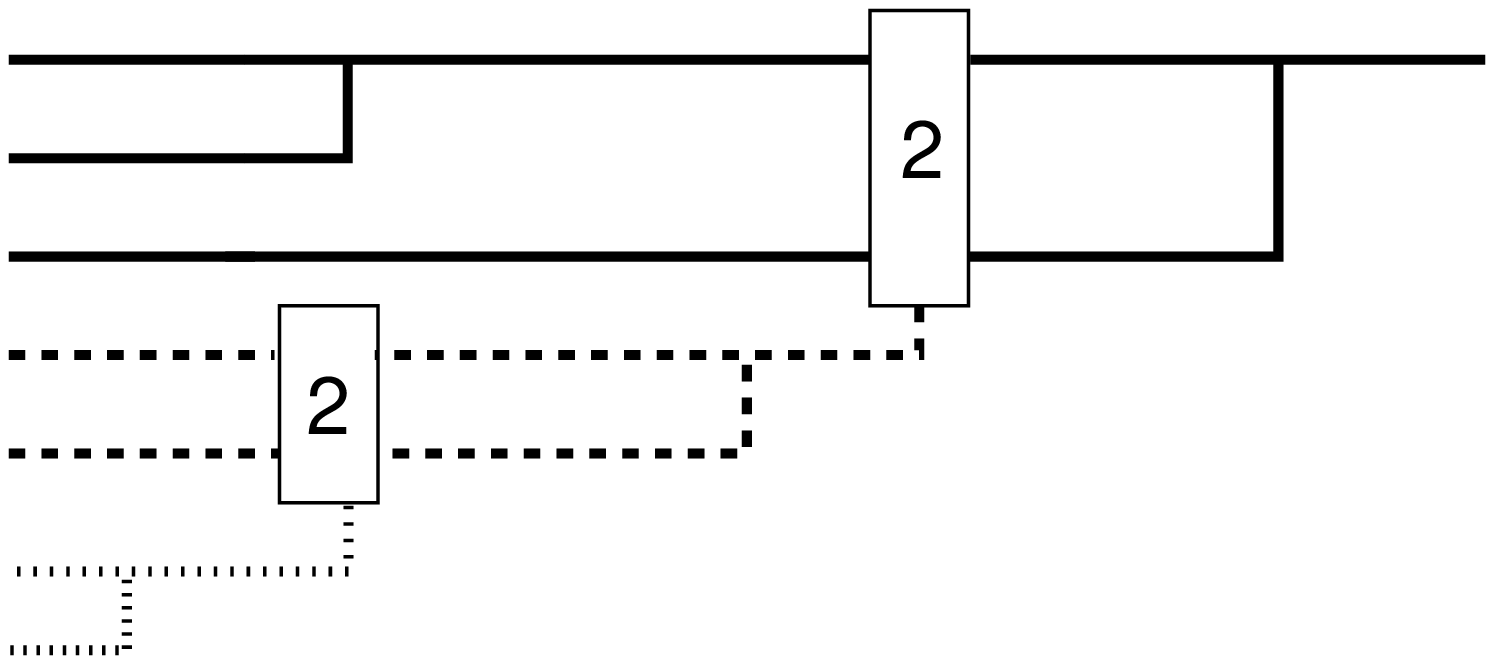}
  \quad
   \includegraphics[width = 0.4\textwidth]{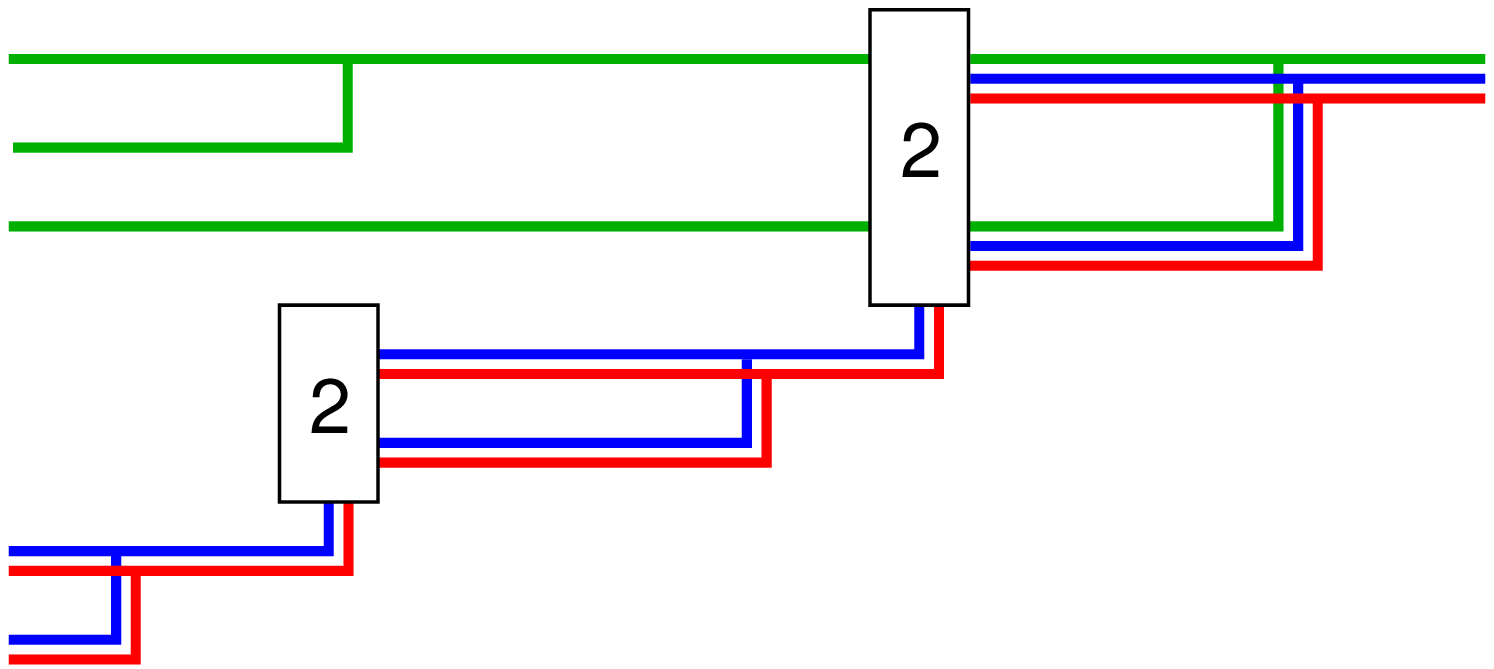}
\end{center}
\hspace{2ex}
 \begin{center}
   \includegraphics[width = 0.4\textwidth]{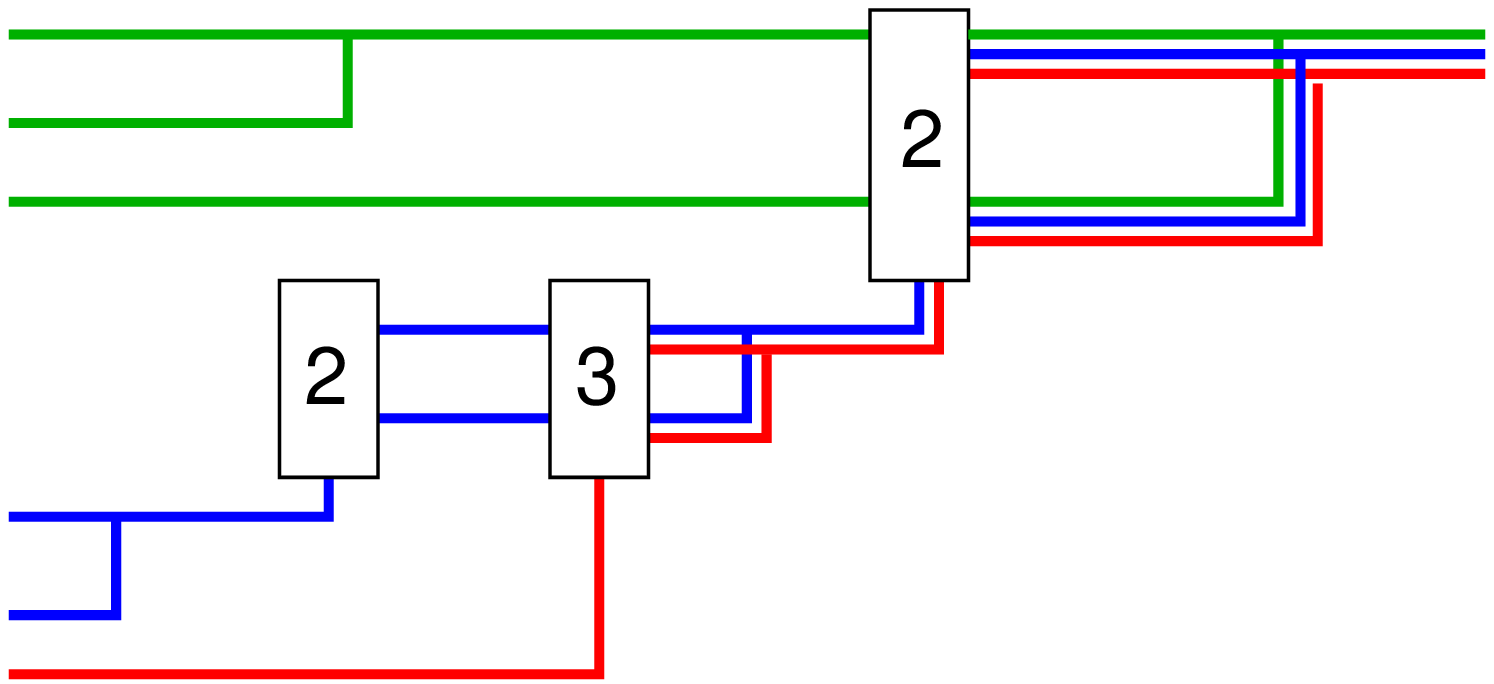}
   \quad
   \raisebox{-1.5ex}{\includegraphics[width = 0.4\textwidth]{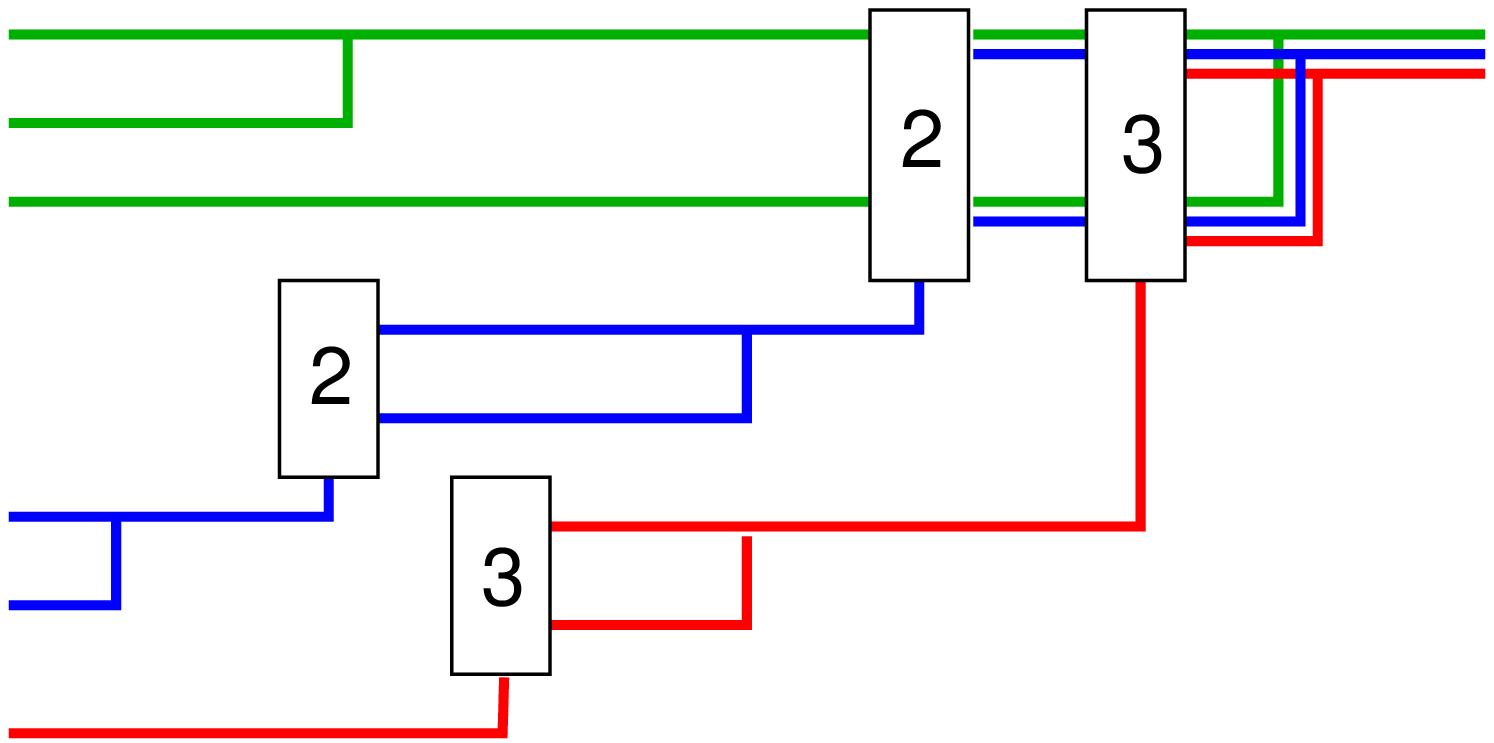}}
  \end{center}
  \caption{\label{essential} 
  Constructing the essential ASRG for three sites with $i^{}_*=1$. Top left: An ASG decorated with a 2-box to which decorated ASGs  are attached repeatedly; solid, dashed, and dotted lines  correspond to steps 1, 2, and 3, respectively. Top right: Labelling and pruning the resulting graph; green, blue and red encode  sites 1,2 and 3, respectively.  Bottom: Adding 3-boxes to the top right graph (two different realisations  bottom left and bottom right). 
  }
  \end{figure}

We now label each line in the graph with the set of sites in the root to which the line is (potentially) ancestral. This will finally allow us to prune away those lines that are not informative for the type of the root, see Figure~\ref{essential} (top right).
We start with the label $A = S$ for the single line at the root. When a branching event occurs to a line labelled $A$, both branches inherit the label. Upon encountering an $i_1$-box, the continuing line is  ancestral to $A \cap C^{(1)}$, while  the line  attached below is ancestral to $A \cap D^{(1)}$. If  $A \cap C^{(1)} = \varnothing$ (this applies to a second and any further $i_1$-box), we prune the continuing line away, because it is neither ancestral to any sites in $A$ at the root, nor does it affect their ancestry. The latter is true because now the same new tail is provided  for \emph{all} potential ancestors of the head, at the same moment; in contrast to the original $\ASRG$, where a new tail may compete with others, see Figure~\ref{collapsedASRG}.  

We finally work up the recursion by decorating  the set of lines potentially ancestral to $D^{(2)}$ with $i_2$-boxes, adding new ASGs, labelling, and pruning in the analogous way, see Figure~\ref{essential} (bottom).  That is, with probability $\ee^{-\varrho^{(2)}t}$, no $i_2$-box appears. With probability $1-\ee^{-\varrho^{(2)}t}$, we add an $i_2$-box, at a time distributed according to $\exponential(\varrho^{(2)})$ conditioned to be $\leq t$. A new ASG labelled $D^{(2)}$ is then attached below, starting with a single line, while the continuing lines  now carry the label $D^{(1)} \cap C^{(2)}$.  If a second $\exponential(\varrho^{(2)})$ waiting time still falls within the remaining time, a second $i_2$-box occurs, with no lines running through it and a single line labelled $D^{(2)}$ starting a new ASG below; and so on until no further $i_2$-box is encountered.

We continue like this until $S^*$ is exhausted. The resulting graph is the \emph{essential \textnormal{ASRG}}. Rather than constructing it via recursion over $S^*$ with successive addition of boxes, labelling, and pruning, it can also be produced in one go in a Markovian manner as follows.
\begin{itemize}
\item Start with a single line labelled $S$. 
\item Every line independently branches at rate $s$; both offspring lines inherit the label of the parent.
\item Every set of lines that carry the same label, say $A$, independently receives an  $i$-box at rate $\varrho_i$ for every $i \in S^*$ with $A \cap D_i \neq \varnothing$, upon which either of the following happens.
\begin{itemize}
\item If $A \cap D_i \neq A$, the lines continue through the box and change their labels to $A \cap C_i$; a single new  line labelled $A \cap D_i$ starts below the box.
\item If $A \cap D_i = A$, no lines continue through the box and  a single new  line labelled $A$ starts below the box.
\end{itemize}
\item Stop when the time horizon $t$ is reached.
\end{itemize}

Note that the resulting graph may be conceived as a collection of (conditionally) independent ASGs, each with its own label, and  joined together by recombination boxes. 
It is now easy to see that all the relevant information can be condensed into a \emph{weighted partitioning process}, namely a Markov process in continuous time that  holds, at any time, an interval partition $\cA$ of $S$ into the blocks $A \in \cA$ of potentially ancestral sites, together with weights $v_A^{}$ giving the number of lines in the respective ASGs. This will be formalised in the next section.

\section{Duality}
\label{sec:duality}
For the genealogical proof of the recursive solution in Theorem~\ref{rekursion}, we relied on the graphical construction, which implicitly assumes a duality between the $\ASRG$ and the solution  of the SRE. 
Since the $\ASRG$ is somewhat unwieldy from a technical standpoint, our next goal is to construct a simpler dual process. Let us begin with our definition of duality for Markov processes, which is a straightforward extension of the standard concept (see \cite[Ch.~3.4.4]{duality} or \cite{kurtjansen} for   thorough expositions, and \cite{Moehle} for an early application to  population genetics).

\begin{definition}\label{duality}
Let $ \cX =(\cX_t)_{t \geq 0}$ and $\cY=(\cY_t)_{t \geq 0}$ be continuous-time Markov processes with state spaces $E$ and $F$, respectively. $\cX$ and $\cY$ are said to be \emph{dual} with respect to some bounded measurable function $H: \; E \times F \to \RR^d$  if 
\begin{equation*}
\EE [ H(\cX_t,y) \mid \cX_0 = x] = \EE [ H(x,\cY_t) \mid \cY_0 = y]
\end{equation*}
holds for all $t \geq 0, x \in E$, and $y \in F$. Furthermore, $H$ is referred to as a \emph{duality function} for $\cX$ and $\cY$ and we abbreviate the duality by $(\cX,\cY,H)$.
\end{definition}

\begin{remark}\label{rem:duality_d}
  The slight extension of the standard concept consists in allowing for an $\RR^d$-valued duality function instead of the usual real-valued $H$. This is, of course, equivalent to introducing a family of $d$ real-valued duality functions. It touches on the interesting problem of finding \emph{all} duality functions  for a given pair of Markov processes. The corresponding \emph{duality space} has been introduced in \cite{Moehle} and investigated in \cite{Moehle13}. \hfill $\diamondsuit$
  \end{remark}

Motivated by our observation at the end of Section~\ref{sec:interlude}, we now define a suitable dual process for $\omega$, and a corresponding duality function. More precisely, we will find three different processes dual to $\omega$, each providing different insight; namely, the weighted partitioning process, a family of  Yule processes with initiation and resetting, and a family of \emph{initiation processes}.

\subsection{The weighted partitioning process} 
For the first dual process, we refer back to the essential ASRG, which can be formalised as a weighted partitioning process. Just as in the  case without selection, the partitioning describes how the genotype of a given individual is pieced together from  the types of its ancestors. To include selection, a positive integer (weight) is assigned to each block, denoting the number of lines in the ASG labelled with this block. As in the single-site case (cf. Figure \ref{peckingorder}), the true ancestor will be of deleterious type if and only if \emph{all} potential ancestors are of deleterious type. 
\begin{definition} \label{WPP}
The \emph{weighted partitioning process} (WPP) is a continuous-time Markov chain $(\varSigma,V) = (\varSigma_t, V_t)_{t \geq 0}$ with (countable) state space
\begin{equation*}
F \defeq  \bigcup_{k \geq 0} \big ( \cI_k (S) \times \NN_+^k \big ),
\end{equation*}
where $\cI_k(S)$ denotes the set of all interval partitions of $S$ into exactly $k$ blocks,
and  transitions 
\begin{enumerate}
\item
$(\cA,v) \longrightarrow (\cA,w)$ at rate $s v^{}_A$ if $w^{}_A = v^{}_A + 1$ for some $A \in \cA$ and $w^{}_B = v^{}_B$ for all $A \neq B \in \cA$. 
\item
$(\cA,v) \longrightarrow (\cA \wedge \{C_i,D_i\},w)$ at rate $\varrho_i$ if, for $i \in S^*$ and the unique $A \in \cA$ with $A \not\subseteq C_i,D_i$, $w^{}_{A \cap C_i} = v^{}_A$, $w^{}_{A \cap D_i} = 1$,  and $w^{}_B = v^{}_B$ for all $A \neq B \in \cA$.
\item
$(\cA,v) \longrightarrow (\cA,w)$ at rate $\varrho^{A \cup \{i_\ast\}}_{\min(A)}$ if, for some  $A \in \cA$,  $w^{}_A = 1$ and $w^{}_B = v^{}_B$ for all $B \neq A$ (the minimum is in the sense of $\preccurlyeq$).
\end{enumerate}
\end{definition}
Note that transition (3) is silent if $w^{}_A = w^{}_B = 1$. 

These transitions are a straightforward translation of the dynamics of the essential $\ASRG$ at the end of Section~\ref{sec:interlude}; clearly, $(\varSigma_t,V_t)=(\cA,v)$ represents the set of ASGs present at time~$t$, where each block $A$ of $\cA$ corresponds to one ASG  with  $v^{}_A$ lines. For every $i \in S^*$, every $A$ splits into $A \cap C_i$ and $A \cap D_i$  at rate $\varrho_i$ independently of all other blocks. If this split is nontrivial, then $A \cap C_i$ inherits the weight of $A$ (reflecting the continuing lines), while the weight of $A \cap D_i$ is set to 1 (reflecting the new ASG attached below  and starting with a single line); this gives transition (2). If $A \subseteq C_i$, nothing happens. If  $A \subseteq D_i$, the weight is reset to 1 (again reflecting the new ASG attached below);  note that this happens whenever the split leaves $A$ intact but  separates it from the selected site, which gives rise to the total rate of $\varrho^{A \cup \{i_\ast\}}_{\min(A)}$ in transition (3). Note also that the marginal $\varSigma$ is the partitioning process of Example~\ref{example_purereco}. Independently of everything else, every block experiences  branching at rate $s$ (transition (1)).
Based on the WPP, we now define the corresponding candidate for our duality function.

\begin{figure}[ht]
\includegraphics[width = 0.7\textwidth]{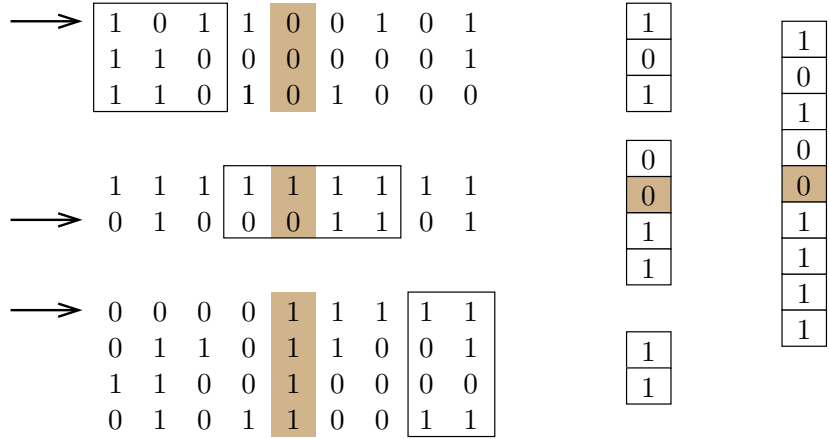}
\caption{\label{fdualityfunction}
Illustration of the duality function $H$ for a WPP in state $(\cA,v)$, where  $\cA = \{\{1,2,3\},\{4,5,6,7\},\{8,9\}\}$ and $v^{}_{\{1,2,3\}} = 3$, $v^{}_{\{4,5,6,7\}} = 2$, and $v^{}_{\{8,9\}} = 4$. The selected site is $i_*=5$ and highlighted in light brown. As prescribed by $(\cA,v)$, we sample $3$ potential ancestors (displayed horizontally on the left) for the first, $2$ for the second, and $4$ for the third block of sites, all i.i.d.\ according to $\nu$. The true ancestor (marked by an arrow) is then sampled uniformly at random from all individuals of beneficial type within the respective samples, except in the case of the third block, since there are no individuals of beneficial type. The resulting marginal types for the individual blocks (middle) are then merged into the sequence on the right. The distribution of this sequence is $H(\cA,v;\nu)$. 
}
\end{figure}

\begin{definition}\label{dualityfunction}
For an interval partition $\cA$ of $S$, associated weights $v \defeq (v^{}_A)^{}_{A \in \cA}$, and $\nu \in \cP(X)$,  we define
\begin{equation*}
H(\cA,v;\nu) \defeq \bigotimes_{A \in \cA} \pi^{}_{A} . \Big ( \big (1-f(\nu) \big )^{v^{}_A} d(\nu) + \big (1 - \big(1-f(\nu) \big)^{v^{}_A} \big )b(\nu) \Big ).
\end{equation*}
\end{definition}

The function $H$ has the following meaning, which is illustrated in Figure \ref{fdualityfunction}. For a given~$(\cA,v)$ and every $A \in \cA$, we sample one sequence according to $\nu$  for each of the $v^{}_A$ leaves of an ASG. The type at the root of this ASG is then distributed according to $b(\nu)$  (according to $d(\nu)$) if at least one of the leaves (none of the leaves) carries a beneficial type, just as in the case of pure selection in Example \ref{example_puresel}. Finally, the sequence at the root of the ASRG is pieced together by taking, for every $A \in \cA$, the sites in $A$ from the type delivered by the ASG corresponding to $A$. The resulting sequence has distribution $H(\cA,v;\nu)$; note that $H(\cA,v;\nu)$ may be understood as a probability vector on $X$, that is, a vector in $\RR^{2^n}$. Before proving the resulting duality, let us  proceed to  a more convenient representation of the WPP.

\subsection{The Yule process with initiation and resetting.}
 Since we are only dealing with single-crossover recombination (and, therefore, only interval partitions), we will take advantage of the following one-to-one correspondence between (weighted) partitions and assignments of nonnegative integers to the sites (see Figure         
\ref{illustration_of_encoding}). Let a vector $m = (m_k)_{1 \leq k \leq n}$ of non-negative integers with $m_{i_*} > 0$ be given. We then obtain an (interval) partition by the rule that two sites $i \prec j$ belong to the same block if and only if $m_k = 0$ for all $i \prec k \preccurlyeq j$; intuitively, the nonzero integers tell us where to chop up the sequence. We obtain in this way a partition $\cA$ in which, for each block $A \in \cA$, $m_{\min(A)}>0$, while $m_i=0$ for $\min(A) \neq i \in A$ (where the minimum is with respect to $\prec$, and is unique since $\cA$ is an interval partition). We then assign a weight to block $A$ by setting $v^{}_A \defeq m_{\min(A)}$.  Likewise, we may encode a weighted partition as an integer vector $m$ by assigning the weight of each block to its minimal site and $0$ to all others. Since $i_\ast$ is the unique minimal element of $S$, one always has $m_{i_*}>0$. Explicitly,
$m^{}_{i_*}  = v^{}_A$ for the unique $A$ that contains $i_* $ and, for $i \neq i_*$, 
\[
m_i = \begin{cases} 0, & \text{if } \overleftarrow i \text{and } i \text{ are in the same block,  } \\
                              v_A \text{ for the unique } A \text{ that contains } i, & \text{otherwise,}
         \end{cases}
\]
with $\overleftarrow i$ as in Definition \ref{porder}. The new encoding allows us to rewrite $H$ of Definition~\ref{dualityfunction} in a convenient way, where we also take advantage of the formalism introduced in Section \ref{sec:solution}. 

\begin{figure}[tb]
\setlength{\unitlength}{1mm}
\qquad
  \begin{picture}(12,15)
    \put(-30,0){$m$}
    \put(-30,5){$\mathcal{A}$}
    \put(-30,10){$v$}
    \put(-15.3,10){$3$}
    \put(-3.2,10){$4$}
    \put(9.5,10){$1$}
    \put(22,10){$2$}
    \put(-23,5){$\{\{ 1,2,3\}$,}
    \put(-7,5){$\{ 4,\textcolor{Tan}{5}\}$,}
    \put(4,5){$\{ 6,7,8\}$,}
    \put(18,5){$\{ 9,10\}\}$}
    \put(-19,0){$0$}
    \put(-15.5,0){$0$}
    \put(-12,0){$3$}
    \put(-5,0){$0$}
    \put(-1.7,0){$4$}
    \put(6,0){$1$}
    \put(9.8,0){$0$}
    \put(13,0){$0$}
    \put(20.2,0){$2$}
    \put(25,0){$0$}
    \end{picture} \\[4mm]
\caption{\label{illustration_of_encoding}
Encoding  a weighted partition (top) by an integer vector (bottom). The selected site is light brown. }
\end{figure}

\begin{lemma}\label{rewritingH}
Let $H$ be as in Definition~\textnormal{\ref{dualityfunction}}. For $m \in \NN_0^S$ with $m_{i_\ast} > 0$, let  $\big (\cA(m),v(m) \big )$ be the weighted partition associated with $m$, and define

\begin{equation*}
 \cH(m,\mu) \defeq H\big((\cA(m),v(m)),\mu \big).
\end{equation*}
Then, one has
\begin{equation}\label{cH}
\cH(m,\mu) = \bigboxtimes_{i \in S}  h(m_i,\mu)^{D_i}
\end{equation}
where
\begin{equation}\label{littleh}
h(k,\mu) \defeq  \big (  1-f(\mu) \big )^k d(\mu) + \big (1 - \big (1-f(\mu) \big)^k \big )b(\mu)   
\end{equation}
for $k \neq 0$ and $h(0,\mu) \defeq 1$. The factors are ordered nondecreasingly with respect to  $\preccurlyeq$. \qed
\end{lemma}

\begin{remark}\label{productordering}
When using the product sign $\small{\bigboxtimes}$ for products of elements of $\fA(X)$ indexed by $S$, we always understand the factors to be ordered nondecreasingly.
\end{remark}

\begin{remark}
  At this point, it becomes clear that the special role of $D_{i_*}=S$  in the definition of the $D_i$ (see Remark~\ref{rem:awkward}) makes perfect sense.  Indeed,  \eqref{cH} shows that the contributions to the sequence at the root of the $\ASRG$ come from the  ASGs associated to the `new' tails $D_i$ that are attached to the original one corresponding to $D_{i_*}=S$. This will become even more evident in the context of the initiation process, see Eq.~\eqref{genmoments} and Fig.~\ref{initprocessfigure}.
  \end{remark}

\begin{proof}
Recall that, by the minimality of the selected site, we have $C_{i_\ast} = \varnothing, D_{i_\ast} = S$ and therefore $\cH(m,\mu) = h(m_{i_\ast},\mu)$ if $m_i = 0$ for all $i \neq i_\ast$. In all other cases, let $i$ be a maximal site with $m_i \neq 0$. The definitions of $\cH$ and $H$ then entail
\begin{equation*}
\cH(m,\mu) = \cH(m',\mu)^{C_i} \otimes h(m_i,\mu)^{D_i} = \cH(m',\mu) \smallboxtimes h(m_i,\mu)^{D_i},
\end{equation*}
where $m'$ is obtained from $m$ by setting $m_i$ to zero. The claim then follows via induction.
\end{proof}

The new encoding also allows us to represent the WPP as a collection of $n$ independent Yule processes with initiation and resetting.
In the  case $s=0$, this is similar to the representation of interval partitions in \cite{recoreview} in terms of the sets of breakpoints. 
\begin{definition}\label{YPIR}
A \emph{Yule process with initiation and resetting} (YPIR) with \emph{branching rate} $s > 0$, \emph{initiation rate} $\varrho \geq 0$, and \emph{resetting rate} $r \geq 0$ is a continuous-time Markov chain on $\NN_{\geq 0}$ with transitions
\begin{equation*}
\begin{array}{lll}
\text{(Y)} & k \to k+1 & \text{ at rate } s k \text{ for } k > 0, \\
\text{(I)} & 0 \to 1 & \text{ at rate } \varrho, \\
\text{(R)} & k \to 1  & \text{ at rate } r  \text{ for } k > 0.
\end{array}
\end{equation*} 
Note that  transition (R) is silent if $k=1$. 
\end{definition}

Given the one-to-one correspondence between $(\cA,v)$ and $m$, it is then easy to see that $(\varSigma,V)$ is equivalent to a collection $ M =(M_{i})_{i \in S}$ of independent YPIRs. Here, $M_{i^{}_*}=(M_{i^{}_*,t})_{t \geq 0}$ is a basic Yule process with branching rate $s>0$, that is, the degenerate case of a YPIR with initiation and resetting rates $\varrho_{i_*} \defeq r_{i_*} \defeq 0$;  for $i \neq i_\ast$, $M_{\!i}=(M_{i,t})_{t \geq 0}$ is a  YPIR  with branching rate $s$, initiation rate $\varrho^{}_i$ and resetting rate
\begin{equation}\label{resrates}
r^{}_i \defeq \sum_{\ell \preccurlyeq i} \varrho^{}_\ell;
\end{equation}
note, in particular, that $r_i \geqslant \varrho_i$. Indeed, the equivalence is clear since the transitions of $(\varSigma,V)$ and $M$ can be matched in a unique way; compare Definitions~\ref{WPP} and \ref{YPIR}.
Note that $r_i$ is the total rate at which $i$ is separated from the selected site; it may be understood as the marginal recombination rate $r_i = \varrho^{\{i,i_*\}}_i$, cf.\ \eqref{margrecorates}. 

Note that the Yule process $K$ (cf. Example~\ref{example_puresel}) has the law of $M_{i_*}$. Let us recapitulate from \cite{BaakeCorderoHummel} the duality  for the pure selection equation, which is a slight extension of Example~\ref{example_puresel}.
\begin{prop} \label{omeganullduality}
Let $K$ be a Yule process with branching rate $s$.  For $k \geq 1$ and $\omega \in \cP(X)$, define $h(k,\omega)$ as in Eq.~\eqref{littleh}.
Then, 
\begin{equation*}
h \big (k,\varphi^{}_t(\mu) \big ) = \EE \big (h(K_t,\mu) \mid K_0 = k \big ),
\end{equation*}
where $\varphi$ is the selection semigroup.
\end{prop}

\begin{proof}
  Combining Eqs.~\eqref{littleh},  \eqref{condunchanged} and \eqref{unfitviayule}, one gets
  \[
    \begin{split}
    h \big (k,  \varphi^{}_t(\mu) \big ) & =
    \big ( \EE \big [ \big ( 1 - f(\mu) \big )^{K_t} \mid K_0 = 1 \big ] \big )^k d \big (\varphi^{}_t(\mu) \big) \\
    & \hphantom{=} +   \big (1 - \big ( \EE \big [ \big ( 1 - f(\mu) \big )^{K_t} \mid K_0 = 1 \big ] \big )^k \big ) b \big (\varphi^{}_t(\mu) \big) \\ & = \EE \big (h(K_t,\mu) \mid K_0 
     = k \big ),
    \end{split}
    \]
where the last step follows from the fact that a collection of $m$  independent Yule processes, each started with a single line, is equivalent to a Yule process started with $k$ lines. 
\end{proof}

Let us still postpone the duality result in the case with recombination to the next section, since the proof is most convenient on the basis of the initiation process. 

\subsection{The initiation process.}Let us first gain some intuition by representing the duality function from Lemma \ref{rewritingH} in terms of products of elements of the selection semigroup at various times. To this end, recall  from Proposition \ref{omeganull} that $\varphi_t(\nu)$ is, for all $\nu$ and $t$, a convex combination of the conditional type distributions $d(\nu)$ and $b(\nu)$,
and so is $h(k,\nu)$ for all $k \geqslant 1$, see Eq.~\eqref{littleh}.  Since  $f(\varphi_t(\nu))$  is strictly increasing in $t$ (cf.\ Proposition~\ref{omeganull}), there exists, for all $k \geq 1$ and $s > 0$, a unique $\theta(k) \in \RR$ such that $\big (1-f(\nu) \big )^k = 1-f \big (\varphi^{}_{\theta(k)} (\nu) \big )$ and thus,
\begin{equation}\label{h_phi}
h(k,\nu) = \varphi^{}_{\theta(k)} (\nu).
\end{equation}
Note that $\theta(1) = 0$ since $h(1,\nu) = \nu = \varphi^{}_0(\nu)$. Then, setting $\theta(0) \defeq \Delta$ and $\varphi_\Delta (\nu) \defeq 1$ for all $\nu$ (in line with $h(0,\bs{\cdot}) = 1$ in Lemma \ref{rewritingH}), we can write, using  Lemma~\ref{rewritingH},
\begin{equation}\label{plusplusplus}
\cH(m,\nu) = \bigboxtimes_{\substack{i \in S}} h(m^{}_i,\mu)^{D_i} = \bigboxtimes_{\substack{i \in S}} \varphi^{}_{\theta(m_i)} (\mu)^{D_i} \eqdef \cG \big (\theta(m),\mu \big ),
\end{equation}
where $\theta(m) \defeq (\theta(m_i))_{i \in S}$. 
More generally, this leads to the ansatz
\begin{equation}\label{genmoments}
\cG(\theta,\nu) \defeq \bigboxtimes_{\substack{i \in S}} \varphi^{}_{\theta_i}(\nu)^{D_i}
\end{equation}
for a third (putative) duality function. Here, $\theta = (\theta_i)_{i \in S} \in \RR_{\geq 0}^{i_\ast} \times (\RR_{\geq 0} \cup \{\Delta\})^{S^*}$ and the symbol $\Delta$ is used to indicate that the factor is absent from the product. Recall that the factors in the product are ordered nondecreasingly w.r.t. $\preccurlyeq$ and note that its value is the same for all such orderings since incomparable factors commute by virtue of being measures defined on projections of the type space $X$ with respect to disjoint subsets of~$S$. 

Recall that $m$ in \eqref{cH} corresponds to a partition of $S$ in which each block is weighted by a positive integer, counting the number of lines in the associated ASG (as part of an essential ASRG, see Section~\ref{sec:interlude}). Similarly, $\theta$ in Eq.~\eqref{genmoments} also encodes a partition of $S$ (the role of $0$ now being played by $\Delta$), only this time, the blocks are not weighted by the number of lines in the associated ASGs, but by their runtimes (again, seen as part of an essential ASRG). In the sampling step, we average over all realisations of the ASG with the indicated runtime, and thus obtain $\cG$ from $\cH$ by replacing
the factors $h(m_i,\nu)$ in $\cH(m,\nu)$ by 
\begin{equation*}
\varphi^{}_{\theta_i}(\nu) = \EE [h(K_{\theta_i},\nu) \mid K_0 = 1];
\end{equation*}
\begin{figure}
\begin{tikzpicture}[scale=0.9]

\draw[{Straight Barb}-] (4,1) -- (8,1);
\draw(6,1.3) node{\small{time}};

\draw (0,0.6) node{$\scriptstyle{t_4^{}}$};
\draw (2.76,0.6) node{$\scriptstyle{t_3^{}}$};
\draw (5.26,0.6) node{$\scriptstyle{t_2^{}}$};
\draw (8.8,0.6) node{$\scriptstyle{t_1^{}}$};
\draw (12,0.6) node{$\scriptstyle{0}$};

\draw[dashed] (12,0.3) -- (12, -3.5);
\draw (12,-3.8) node{$\scriptstyle{(0,\Delta,\Delta)}$};
\draw[dashed] (12,-4.1) -- (12,-5);
\draw (12, -5.3) node{$\scriptstyle{\varphi_0^{}}$};

\draw[dashed] (8.8, 0.3) -- (8.8, -4.5);
\draw (8.8,-4.8) node{$\scriptstyle{(t_1^{},0,\Delta)}$};
\draw[dashed] (8.8, -5.1) -- (8.8,-5.7);
\draw (8.8,-6) node{$\scriptstyle{\varphi_{t_1}^{} \smallboxtimes \varphi_0^{D_2}}$};

\draw[dashed] (5.26,0.3) -- (5.26, -3.5);
\draw (5.26,-3.8) node{$\scriptstyle{(t_2^{},t_2^{} - t_1^{},0)}$};
\draw[dashed] (5.26,-4.1) -- (5.26,-5);
\draw (5.26, -5.3) node{$\scriptstyle{\varphi_{t_2}^{} \smallboxtimes \varphi_{t_2 - t_1}^{D_2} \smallboxtimes \varphi_0^{D_3}}$};

\draw[dashed] (2.76, 0.3)  -- (2.76, -4.5);
\draw (2.76,-4.8) node{$\scriptstyle{(t_3^{},0,t_3^{} - t_2^{})}$};
\draw[dashed] (2.76, -5.1) -- (2.76,-5.7);
\draw (2.76,-6) node{$\scriptstyle{\varphi_{t_3}^{} \smallboxtimes \xcancel{\varphi_{t_3 - t_1}^{D_2}} \smallboxtimes \varphi_0^{D_2} \smallboxtimes \varphi_{t_3 - t_2}^{D_3}}$};

\draw[dashed] (0,0.3) -- (0, -3.5);
\draw (0,-3.8) node{$\scriptstyle{(t_4^{},t_4^{} - t_3^{}, t_4^{} - t_2^{})}$};
\draw[dashed] (0,-4.1) -- (0,-5);
\draw (0, -5.3) node{$\scriptstyle{\varphi_t^{} \smallboxtimes \varphi_{t_4 - t_3}^{} \smallboxtimes \varphi_{t_4 - t_2}^{D_3}}$};

\draw[ultra thick, gre] (12,0) -- (0,0) node[left, black]{$\scriptstyle{t_4^{}}$};
\draw[ultra thick, blue] (9,-0.08) -- (12,-0.08);
\draw[ultra thick, red] (9,-0.16) -- (12,-0.16);

\draw[fill = white] (8.52,0.16) rectangle (9,-0.32) node[midway]{$\scriptstyle{2}$};
\draw[ultra thick, blue] (8.72,-0.32) -- (8.72, -1.12);
\draw[ultra thick, red] (8.8,-0.32) -- (8.8, -1.2);

\draw[ultra thick, blue] (8.72, -1.12) -- (5.5,-1.12);
\draw[ultra thick, red] (8.8,-1.2) -- (5.5,-1.2);
\draw[fill = white] (5.02, -0.92) rectangle (5.5,-1.4) node[midway]{$\scriptstyle{3}$};

\draw[ultra thick, blue] (5.02,-1.12) -- (3,-1.12);
\draw[fill = white] (2.52,-0.92) rectangle (3,-1.4)  node[midway]{$\scriptstyle{2}$};
\draw[ultra thick, blue] (2.76,-1.4) -- (2.76, -2.16) -- (0, -2.16) node[left, black] {$\scriptstyle{t_4^{} - t_3^{}}$};

\draw[ultra thick, red] (5.26,-1.4) -- (5.26,-3) -- (0,-3) node[left, black] {$\scriptstyle{t_4^{} - t_2^{}}$};


\end{tikzpicture}
\vspace{2mm}
\caption{
\label{initprocessfigure}
A realisation of the essential ASRG, where every ASG is collapsed into a single line. It describes the evolution of a partitioning process whose blocks are weighted by the time  since the corresponding ASG was attached. The colours are  as in Figure \ref{essential}: green, blue and red  for site 1,2 and 3; as before, the first site is selected. Below the graph, we indicate the evolution of the associated collection of initiation processes $\Theta$. At the bottom, we see how the function $\cG(\Theta_t,\cdot)$, defined in Eq.~\eqref{genmoments},  evolves in time. Every factor corresponds to a different line, and attachment of a new line due to an $i$-recombination event corresponds to right-multiplication by $\varphi_0^{D_i}$; subsequently, the time index in each factor evolves on its own. Notice the cancellation  at time $t_3$; it corresponds to the discontinuation of the line at the box representing recombination and the reset of the second component of $\Theta$, due to $\{2\} \cap D_2 = \{2\}$. 
}
\end{figure}
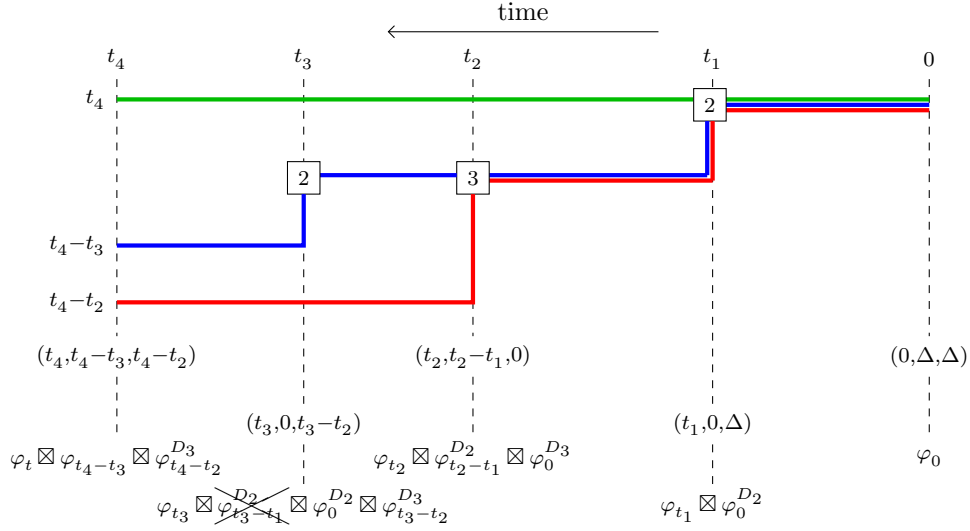
this will later make the connection to the transformation \eqref{h_phi}.

We now give an informal description of the \emph{initiation process} $\Theta$, which will take the role of the YPIR.
It is a continuous-time Markov process, and its transition rates relate to that of the YPIR as follows.
As $\Delta$ takes the role of $0$, the transition (I) (initiation) in Definition \ref{YPIR} corresponds to a transition from $\Delta$ to $0$. Similarly, as $0$ takes the role of $1$, a reset (R) (to $1$) of the YPIR corresponds to a reset (to $0$) of the initiation process. 
Keeping in mind that (Y) describes the branching of the ASG (and that we now only want to record its runtime), we replace these random jumps by a deterministic and continuous increase. Thus, $\Theta_t$ is either $\Delta$, signifying that it has not yet been initiated, or its value is just the time that has passed since the last reset. Finally, when no resetting occurs, we have $\Theta_t = \Theta_0 + t$.
 
This can be condensed into the following definition; for an illustration, see Fig.~\ref{initprocessfigure}.
\begin{definition}\label{theta}
The \emph{initiation process} with \emph{initiation rate}  $\varrho \geq 0$  and \emph{resetting rate} $r \geq 0$ is the 
continuous-time Markov process with values in $\RR_{\geq 0} \cup \{\Delta\}$ and generator mapping $u \in C^1(\RR)$ to $\bar u$, defined via
\begin{equation}\label{generators}
\begin{split}
\bar u (t) & \defeq \dot{u}(t) + r \big (u(0) - u(t) \big ) \quad \text{for  }  t \in \RR_{\geq 0}, \\
\bar u(\Delta) & \defeq \varrho \big (u(0) - u(\Delta) \big ).
\end{split} 
\end{equation}
We define a collection $\Theta = (\Theta_{i})_{i \in S}$ of independent initiation processes where $\Theta_i=(\Theta_{i,t})_{t \geq 0}$  has initiation rate $\varrho_i$ and resetting rate $r_i$ (cf. \eqref{resrates}). In particular, since $\varrho^{}_{i_*} = r^{}_{i_*} = 0$, all stochastic contributions in Eq.~\eqref{generators} vanish for this choice, and what remains is a purely deterministic drift, that is,  $\Theta_{i_*,t}  = t + \Theta_{i_*,0}$. We denote by $\cL_i$ the generator of $\Theta_i$. Furthermore, $\cL \defeq \sum_{i \in S}\cL_i$, where $\cL_i$ acts on the $i$-th component of the argument.
\end{definition}
Note that $\Theta$ shares the parameters $\varrho_i$ and $r_i$ with $M$, but  does not depend on $s$. Rather, for any given $s$, we will see that $\Theta$ and $M$ are related at the level of an expectation.
First, we prove the duality $(\omega,\Theta,\cG)$. From there, we recover  $(\omega,M,\cH)$ and, equivalently, $(\omega,(\varSigma,V),H)$. 

\begin{prop}\label{h_M_theta}
For all $i \in S$, the \textnormal{YPIR} $M_i$  and the initiation process $\Theta_i$ of Definition~\textnormal{\ref{theta}} satisfy
\begin{equation*}
\EE \big ( h(M_t,\nu) \mid M_0 = m \big ) = \EE \big ( \varphi^{}_{\Theta_t}(\nu) \mid \Theta^{}_{0} = \theta(m)  \big )
\end{equation*}
for all $m_i \in \NN_0$ and $t \geq 0$.
\end{prop}
\begin{proof}
It suffices  to show that the left- and right-hand side of the statement solve the same initial value problem (with globally Lipschitz continuous right-hand side). By \eqref{h_phi}, the expressions agree at $t=0$. It remains to be shown that 
\begin{equation*}
\cQ_i h (\bs{\cdot},\nu) (m_i) = \cL_i \varphi^{}_{\bs{\cdot}} (\nu) \big (\theta(m_i) \big ), 
\end{equation*}
where $\cQ_i$ is the generator of $M_i$, and $\cL_i$ that of $\Theta_i$. Comparing Definitions~\ref{YPIR} and \ref{theta}, it is obvious that the transitions from $m_i$ to 1 in the YPIR (at rate $\varrho_i$ if $m_i=0$ and at rate $r_i$ if $m_i>0$) correspond to transitions to $0$ in the initiation process  (at rate $\varrho_i$ if $\Theta_i=\Delta$ and at rate $r_i$ if $\Theta_i \in \RR_{\geqslant 0}$).  The identity \eqref{h_phi} then implies the equality of the corresponding contributions to the left and right-hand side, i.e.

\begin{equation*}
\begin{split}
h(1,\nu) - h(m_i,\nu) & = \varphi^{}_0(\nu)  - \varphi^{}_{\Delta} (\nu) \quad \text{for } m_i=0, \quad \text{and} \\
h(1,\nu) - h(m_i,\nu) & = \varphi^{}_0(\nu) - \varphi^{}_{\theta(m_i)}(\nu) \quad \text{for } m_i>0.
\end{split}
\end{equation*}
Furthermore, it is a direct consequence of Proposition \ref{omeganullduality} together with \eqref{h_phi} that the time derivative corresponds to branching of the YPIR, that is,
\begin{equation*}
\dot{\varphi}^{}_{\theta(m_i)}(\nu) = \myfrac{\dd}{\dd t} \EE \big ( h(K_t, \nu) \mid M_0 = m_i) \big ) |_{ t = 0} = s m_i \big (h(m_i + 1, \nu) - h(m_i,\nu)       \big)
\end{equation*}
by the Kolmogorov backward equation for the Yule process.
\end{proof}

Returning now to $\cH$ and $\cG$, we obtain immediately, by independence:
\begin{coro}\label{M_theta}
The families $M$ and $\Theta$ of independent \textnormal{YPIRs} and initiation processes satisfy
\begin{equation*}
\EE \big ( \cH(M_t,\nu) \mid M_0 = m \big ) = \EE \big ( \cG(\Theta_t,\nu) \mid \Theta_{0} =  \theta(m)  \big )
\end{equation*}
for all $m \in \NN_0^n$ and $t \geqslant 0$.
\end{coro}

We are now set to state the main result of this section,  the duality $(\omega,\Theta,\cG)$.
\begin{theorem}\label{protoduality}
Let $\Theta$ be the family of independent initiation processes introduced in Definition \textnormal{\ref{theta}}. 
Then, with $\cG$ as in \eqref{genmoments}, we have, for all $\nu \in \cP(X)$ and all $\theta \in \RR_{\geq 0}^{i_\ast} \times (\RR_{\geq 0} \cup \{\Delta\})^{S^*}$,
\begin{equation*}
\cG \big (\theta,\psi_t(\nu) \big ) = \EE \big ( \cG(\theta, \omega^{}_t) \mid \omega^{}_0=\nu) = \EE ( \cG(\Theta_t, \nu) \mid \Theta_0 = \theta \big ),
\end{equation*}
where $\psi= (\psi_t)_{t \geq 0}$ is the flow of the \textnormal{SRE} introduced in Definition \textnormal{\ref{hierarchy}}. 
\end{theorem}

\begin{proof}
The first equality is clear because $\psi$ is deterministic. 
For the proof of the second equality (the actual duality relation), it will be useful to think of the solution of the SRE \eqref{main} as a deterministic Markov process with generator $\tilde{\Psi} = \Psiseltilde + \Psirecotilde$ given by
\begin{equation*}
\begin{split}
\tilde{\Psi} f(\nu) & \defeq \myfrac{\dd}{\dd t} f \big (\psi_t(\nu) \big ) |_{t = 0} = \myfrac{\dd}{\dd t} f\big (\nu + t \Psisel(\nu) + t \Psireco(\nu) \big )_{t = 0} \\ 
& = \myfrac{\dd}{\dd t} f \big (\nu + t \Psisel(\nu) \big ) |_{t = 0} + \myfrac{\dd}{\dd t} f \big (\nu + t \Psireco(\nu)\big ) |_{t = 0} \\[1mm]
& \eqdef \Psiseltilde f (\nu) + \Psirecotilde f (\nu)
\end{split}
\end{equation*}
for all $f \in C^{1}(\cP(X))$.

As in the proof of Proposition~\ref{h_M_theta}, we will show that the left and right-hand side satisfy the same initial value problem. As the values at $t = 0$  agree, see Eq.~\eqref{h_phi}, it suffices to show that 
\begin{equation}\label{generatorengleich}
\tilde{\Psi} \cG (\theta,\bs{\cdot})(\nu) = \cL \cG(\bs{\cdot},\nu)(\theta)
\end{equation}
for all $\nu \in \cP(X)$ and all $\theta \in \RR_{\geq 0}^{i_\ast} \times (\RR_{\geq 0} \cup \{\Delta\})^{S^*}$. (Indeed, if \eqref{generatorengleich} is satisfied, it trivially applies to all components of the $\RR^{2^n}$-valued function $\cG$ and thus establishes duality also in our slightly extended sense; cf.~Remark~\ref{rem:duality_d}.)
First of all, let us note that, since $\tilde{\Psi}$ is a differential operator, we have
\begin{equation}\label{Psi_G}
\tilde{\Psi} \big ( \cG(\theta,\bs{\cdot}) \big )(\nu) = \sum_{\substack{j \in S \\ \theta_j \neq \Delta}} \Big ( \bigboxtimes_{j \not \prec \udo{i} \in S \setminus j } \varphi^{}_{\theta_i}(\nu)^{D_i} \Big ) \smallboxtimes \big ( \tilde{\Psi}(\varphi^{}_{\theta_j})(\nu) \big )^{D_j}   \smallboxtimes  \bigboxtimes_{j \prec \udo{i} \in S} \varphi^{}_{\theta_i}(\nu)^{D_i}
\end{equation}
by the product rule, where the underdot indicates variable (in this case $i$) with respect to which the product is performed; note that since $\varphi_\Delta^{}(\nu) = 1$, factors with $\theta_i = \Delta$ play no role. Hence, in order to evaluate the left-hand side of Eq.~\eqref{generatorengleich}, we only need to compute
$\big ( \tilde{\Psi} (\varphi_{\theta_j}^{})(\nu) \big )^{D_j}$ 
for all $j \in S$ such that $\theta_j \neq \Delta$. 
Clearly,
\begin{equation}\label{selcontrib}
\big (\Psiseltilde(\varphi_{\theta_j}^{})(\nu)\big)^{D_j} = \big (\dot{\varphi}_{\theta_j^{}}(\nu) \big )^{D_j}
\end{equation}
because $\varphi$ is the flow of the pure selection equation. 
For the recombination part, we calculate 
\begin{equation}\label{Psi_G_2}
\begin{split}
\big ( \Psirecotilde & (\varphi_{\theta_j}^{})(\nu) \big )^{D_j} \\
&= \Big ( \myfrac{\dd}{\dd h} \ts \varphi_{\theta_j}^{} \big (\nu + h \Psireco(\nu) \big)|_{ h = 0} \Big )^{D_j} \\ 
&= \Big ( \myfrac{\dd}{\dd h} \ts \varphi_{\theta_j}^{} \big(\nu \smallboxtimes \big(\1 \smallboxplus h \sum_{\ell \in S^*}\varrho_\ell( \nu^{D_\ell} \smallboxminus \1) \big) \big)|_{ h = 0} \Big)^{D_j} \\
&= \Big ( \varphi_{\theta_j}^{}(\nu) \smallboxtimes  \myfrac{\dd}{\dd h}\big(\1 \smallboxplus h \sum_{\ell \in S^*} \varrho_\ell (\nu^{D_\ell} \smallboxminus \1) \big)|_{h = 0} \Big)^{D_j} \\ 
&= \sum_{\ell \in S^*}\varrho^{}_\ell \big ( \varphi_{\theta_j}^{}(\nu)^{D_j} \smallboxtimes \nu^{D_\ell \cap D_j} - \varphi_{\theta_j}^{}(\nu)^{D_j} \big) \\
&= \sum_{\substack{\ell \in S^* \\ \ell \preccurlyeq j}}\varrho_\ell \big (\varphi_0^{}(\nu)^{D_j} - \varphi_{\theta_j}^{}(\nu)^{D_j} \big )
+ \sum_{\substack{\ell \in S^* \\ \ell \succ j}} \varrho_\ell \big (\varphi_{\theta_j}^{}(\nu)^{D_j} \smallboxtimes \varphi_0^{}(\nu)^{D_\ell} - \varphi_{\theta_j}^{}(\nu)^{D_j} \big).
\end{split}
\end{equation}
Here, we have used Lemma \ref{equivariance} in the third step,
and in the last that $\varphi^{}_0(\nu) = \nu$ together with  the fact that the sum over sites incomparable to $j$ vanishes because $D_j \cap D_\ell = \varnothing$ if $\ell$ is incomparable to $j$. To simplify the first sum, we took advantage of the fact that $\ell \preccurlyeq j$ implies $D_j \subseteq D_\ell$ together with the cancellation rule from Proposition~\ref{algebraprops}.  Similarly, $\ell \succ j $ implies $D_\ell \subseteq D_j$, which simplifies the second sum. 
Inserting \eqref{Psi_G_2} and \eqref{selcontrib} into \eqref{Psi_G} and recalling Eq.~\eqref{resrates},
we have  shown  so far that 
\[
\begin{split}
\tilde{\Psi} \cG & (\theta,\bs{\cdot})(\nu) \\
& = \sum_{\substack{j \in S \\ \theta_j \neq \Delta}} \Big (r_j \big (\cG((\theta_{<j},0,\theta_{>j}),\nu )  - \cG(\theta,\nu) \big) + \myfrac{\partial}{\partial_{\theta_j}} \cG (\theta,\nu) + \sum_{\ell \succ j} \varrho_\ell \big ( \cG_{j,\ell}(\theta,\nu) - \cG(\theta,\nu) \big ) \Big),
\end{split}
\]
where we use the obvious convention that $(\theta_{< j},0,\theta_{>j})$ is obtained from $\theta$ by setting $\theta_j$ to $0$. Furthermore,
$\cG_{j,\ell}(\theta,\nu)$ (for $t_j \neq \Delta$ and $j  \prec \ell$) arises from $\cG(\theta,\nu)$ by inserting the factor  $\varphi_0^{}(\nu)^{D_\ell}$ at the immediate right of 
$\varphi_{\theta_j}^{}(\nu)^{D_j}$. That is, if $\cG(\theta,\nu)$ is of the form $\cG(\theta,\nu) = \kappa \smallboxtimes \varphi_{\theta_j}^{}(\nu)^{D_j} \smallboxtimes \chi$, then
\begin{equation}\label{inserted}
\cG_{j,\ell}(\theta,\nu) = \kappa \smallboxtimes \varphi_{\theta_j}^{}(\nu)^{D_j} \smallboxtimes \varphi_0^{}(\nu)^{D_\ell} \smallboxtimes \chi .
\end{equation}

Hence, if we can show that
\begin{equation}\label{21341356nochzuzeigen}
\sum_{\substack{j \in S \\ \theta_j \neq \Delta}} \sum_{\ell \succ j }\varrho_\ell \big ( \cG_{j,\ell}(\theta,\nu) - \cG(\theta,\nu) \big )= \sum_{\substack{\ell \in S^* \\ \theta_\ell = \Delta}}  \varrho_\ell \big (\cG((\theta_{< \ell},0,\theta_{> \ell}),\nu) - \cG(\theta,\nu) \big),
\end{equation}
it follows that $\tilde{\Psi} \cG(\theta,\bs{\cdot})(\nu) = \sum_{j \in S} \cL_j \cG((\theta_{< j},\bs{\cdot},\theta_{>j}),\nu)(\theta_j) = \cL \cG( \bs{\cdot}, \nu) (\theta)$.

To see Eq.~\eqref{21341356nochzuzeigen}, notice that, if $j \neq \max\{j' \preccurlyeq \ell: \theta^{}_{j'} \neq \Delta \}$) (in particular, this is the case  if $\theta_\ell \neq \Delta$), then  $\cG_{j,\ell}(\theta,\nu)$ is of the form
\begin{equation}\label{siteordered}
\kappa \smallboxtimes \varphi_{\theta_j}(\nu)^{D_j} \smallboxtimes \varphi_0(\nu)^{D_\ell} \smallboxtimes \varphi_{\theta_{j'}}(\nu)^{D_{j'}} \smallboxtimes \chi' 
\end{equation}
for some $j' \preccurlyeq \ell$ due to the site ordering (cf. Remark \ref{productordering}), where $\chi = \varphi^{}_{\theta_{j'}} \smallboxtimes \chi'$.
Since $j' \preccurlyeq \ell$ means $D_{\ell} \subseteq D_{j'}$, \eqref{siteordered}  is equal to
\begin{equation*}
\kappa \smallboxtimes \varphi_{\theta_j}^{}(\nu)^{D_j} \smallboxtimes \varphi_{\theta_{j'}}^{}(\nu)^{D_{j'}} \smallboxtimes \chi' = \cG(\theta, \nu) 
\end{equation*}
by the cancellation rule from Proposition~\ref{algebraprops}. If $j=\max\{j' \preccurlyeq \ell: \theta^{}_{j'} \neq \Delta \}$,  the factors in \eqref{inserted} are ordered strictly nondecreasingly w.r.t.\ $\preccurlyeq$, and no cancellations occur; hence we have $\cG_{j,\ell}(\theta,\nu) =  \cG((\theta_{< \ell }, 0, \theta_{> \ell }),\nu)$. Thus,  we have verified \eqref{21341356nochzuzeigen}.
\end{proof}

\begin{remark}
  \begin{enumerate}[label=(\roman*)]
\item 
Another approach to recover Theorem \ref{protoduality} would be to prove the right multiplicativity for $h(k, \bs{\cdot})$ for $k \geq 1$ by the same argument as in Lemma \ref{equivariance}, and to replace $\varphi_t$ by $h(k, \bs{\cdot})$ in the proof of Theorem \ref{protoduality}.
\item
Note that the particular form of the selection term was not used in the proof of Theorem \ref{protoduality}; the only property required was the second statement in Lemma~\ref{equivariance}. Therefore, the same procedure can be applied to any single-locus model with linked neutral sites. Examples include the deterministic mutation-selection equation, for which the dual process can then be expressed as a collection of independent \emph{pruned lookdown \textnormal{ASGs}} \cite{BaakeCorderoHummel,review} that are initiated and reset at random.
\item
It is also instructive to  relate the  proof of Theorem \ref{protoduality} to the genealogical construction detailed above; see Figure \ref{initprocessfigure}. Recall that the  factors $\varphi^{}_{\theta_j}(\nu)^{D_j}$ in $\cG(\theta,\nu)$ correspond to the different independent ASGs that make up the essential $\ASRG$ of Section~\ref{sec:asrg}, and which are ancestral to different sets of sites. At rate $\varrho_\ell$, $\ell \in S^*$,  each such ASG is hit independently by an $\ell$-box, at which a new ASG is started for the tail. This corresponds to right multiplication of $\varphi_{t_j}(\nu)^{D_j}$ by $\varphi_{t_\ell}(\nu)^{D_\ell}$. Recall that, for  such a multiplication, we had to distinguish the three cases of $j$ being either incomparable to $\ell$, $\ell \preccurlyeq j$ and $\ell \succ j$. In the genealogical picture, these  cases correspond  to the recombination event being either ignored (if $\ell$ and $j$ are incomparable, which entails that the ASG in question is only ancestral to sites in $C_\ell$); a resetting event if $\ell \preccurlyeq j$, which means that the ASG is only ancestral to sites contained in $D_{\ell}$; or an initiation event if $\ell \succ j$, where a new ASG is initiated for the tail. \hfill $\diamondsuit$
\end{enumerate} 
\end{remark}

By Corollary~\ref{M_theta} and \eqref{plusplusplus}, Theorem~\ref{protoduality} also yields the duality of $\omega$ and $M$.

\begin{coro}\label{dual_H}
The family $M$ of \textnormal{YPIRs} and the solution $\omega$ of the \textnormal{SRE} \eqref{main} are dual with respect to $\cH$ of \eqref{defH}, namely 
\begin{equation}\label{Hk}
\begin{split}
\EE \big [ \cH(M_t, \nu) \mid M_0 = m \big ] = \EE \big [ \cH(m,\omega^{}_t) \mid \omega^{}_0 = \nu \big ] = \cH \big (m,\psi_t(\nu) \big )
\end{split}
\end{equation}
for all $\nu \in \cP(X)$ and all initial values $m \in \NN_0^S$ with $m_{i_\ast} > 0$. Here, $\psi$ is the deterministic flow introduced in Definition \textnormal{\ref{hierarchy}}. \hfill \qed
\end{coro}

The following representations analogous to \eqref{Beispiel fuer Dualitaet} for the solution of the  selection-recombination differential equation are now immediate.
\begin{coro}\label{representation}
Let $\omega = \psi (\omega^{}_0)$ be the solution of the \textnormal{SRE} \eqref{main}. Then, for all $t \geq 0$, we have the stochastic representations
\begin{equation*}
\omega_t^{} = \EE \Big [ \cH (M_t,\omega^{}_0) \mid M_{i,0} = \delta_{i,i_\ast}  \text{ for } i \in S \Big ] 
= \EE \Big [ \cG (\Theta_t,\omega^{}_0) \mid \Theta_{i_*,0} = 0, \Theta_{i,0}= \Delta \text{ for } i \in S^* \Big ] 
\end{equation*}
with $\cH$ of  \eqref{defH} and $\cG$ of \eqref{genmoments}.
That is, we average over all realisations of the \textnormal{WPP}, starting from the trivial partition with weight one, as represented by the family of \textnormal{YPIRs}, or the family of initiation processes, started in 0 for $i=i_*$ and started in $\Delta$ for $i \in S^*$. 
\end{coro}

\section{The explicit solution and its long-term behaviour}
\label{sec:asymptotics}
We have just seen that the solution of the SRE  has a stochastic representation in terms of a collection of independent YPIRs. Their semigroups are easily expressed in terms of geometric distributions with random success probability.

\begin{prop}\label{YPIRdist}
Let $i \in S$ and let $M_i$ be a \textnormal{YPIR} with branching rate $s > 0$,  initiation rate $\varrho_i^{} \geq 0$ and resetting rate $r_i^{} \geq 0$. If $r_i^{}>0$, let $T_i$ be a random variable  with distribution $\exponential(r_i^{})$; if $r_i^{}=0$, set $T_i \defeq \infty$ for consistency. The Markov  semigroup $p_i=(p^{}_{i,t})^{}_{t \geq 0}$ corresponding to $M_i$  is then given by
\begin{equation*}
\begin{split}
p_{i,t}(1, \bs{\cdot}) & = \EE \big [\geom (\mathrm{e}^{-s{(T_i \wedge t)}}) \big ],\\
p_{i,t}(0,n_i^{}) & = \int_0^\infty \varrho \ts \mathrm{e}^{-\varrho_i^{} \tau} p_{i,t - \tau} (1,n_i^{}) \dd \tau + \delta_{0,n^{}_i} \mathrm{e}^{-\varrho_i^{} t}, \\
p_{i,t}(m_i^{},n_i^{}) & = \int_0^\infty r_i^{} \mathrm{e}^{-r_i^{} \tau} p_{i, t - \tau} (1,n_i^{}) \dd \tau + \mathrm{e}^{-r_i^{} t} \negbin(m_i^{},\mathrm{e}^{-st}) (n_i^{}), \quad m_i^{} \geq 1,
\end{split}
\end{equation*}
 where $\negbin(m_i^{},\sigma)$ is the negative binomial distribution with parameters $m_i^{}$ and $\sigma$, and we set $p_{i,t}(1,\bs{\cdot}) \equiv 0$ for $t < 0$. 
\end{prop}
\begin{proof}
For the first formula, we argue as in the genealogical proof of Theorem \ref{rekursion}. After the time of  the last resetting event, which follows $\exponential(r_i^{})$, the YPIR experiences no further resetting and hence has the law of a Yule process with branching rate $s$ for the remaining time. The second and third formulae follow from the first by waiting the $\exponential(\varrho_i^{})$ $(\exponential(r_i^{}))$-distributed time until the process initiates (resets); recall that $\negbin(m_i^{},\sigma)$ is the distribution of the number of independent Bernoulli trials (with success probability  $\sigma$) up to and including the $m_i^{}$th success. 
In the degenerate case $r_i^{} = 0$ and $\varrho_i^{} = 0$, the statement reduces to 
\begin{equation*}
p_{i,t}(m_i^{},n_i^{})  =  \negbin(m_i^{},\mathrm{e}^{-st}) (n_i^{}), \quad m_i^{} \geq 1, \,\ p_{i,t}(0,n_i^{}) = \delta_{0,n_i^{}} 
\end{equation*}
which is just the semigroup of the ordinary Yule process. The consistency in the cases where only one of the parameters $\varrho_i^{}$ or $r_i^{}$ vanishes is seen just as easily.
\end{proof}  

Combining Proposition~\ref{YPIRdist} with Corollary \ref{representation} yields a closed expression for the solution.
\begin{coro}\label{explicitsolution}
The solution of the \textnormal{SRE} is given by
\begin{equation*}
\omega_t = p_{i_*,t}  h(\bs{\cdot},\omega_0)^{D_{i_\ast}} (1) \smallboxtimes \bigboxtimes_{i \in S^*} p_{i,t} h(\bs{\cdot},\omega_0)^{D_i}(0),
\end{equation*}
where $p^{}_i = (p^{}_{i,t})_{t \geq 0}$ is the semigroup of $M_i^{}$ as in Proposition~\textnormal{\ref{YPIRdist}}.
\end{coro}

This explicit representation allows us to investigate the long-term behaviour of the solution.
We start with the asymptotics of the semigroup from Proposition \ref{YPIRdist}.

\begin{coro}\label{YPIRasymptotics}
As in Proposition~\textnormal{\ref{YPIRdist}}, let $p_i^{}$ be the Markov semigroup of the \textnormal{YPIR} $M_i$. Then, for all $m_i^{} \geq 0$,
\begin{equation}\label{135243}
p_{i,\infty} (n_i^{}) \defeq \lim_{t \to \infty} p_{i,t}^{} (m_i^{},n_i^{}) =  \EE \big [\geom (\mathrm{e}^{-sT_i}) (n_i^{})  \big ] = \alpha B(n_i^{},\alpha+1), \quad n_i^{}=1,2,\ldots,
\end{equation}
where $T_i^{}$ follows $\mathrm{Exp}(r_i^{})$, $\alpha \defeq r_i^{}/s_i^{}$, and $B$ denotes the beta function. If $\varrho_i^{}=0$, then $p_{i,t}(0,n_i^{}) = \delta_{0,n_i^{}}$, and Eq.~\eqref{135243} applies for $m_i^{} >0$.

\end{coro}

\begin{proof}
Since $M_i$ is irreducible, positive recurrent, and non-explosive (since it is stochastically dominated by a Yule process with branching rate $s$, which is non-explosive), it has a unique asymptotic distribution $p_{i, \infty}$ such that $p_{i,t}(m_i^{},\bs{\cdot})$ converges to $p_{i,\infty}$ for all initial conditions $m_i^{} > 0$.  To see that in fact $p_{i, \infty} = \EE \big (\geom (\mathrm{e}^{-sT_i}) \big )$, it suffices in the case $\varrho_i^{} > 0$ to simply let $t \to \infty$ in $p^{}_{i,t} (1,\bs{\cdot})$  in Proposition \ref{YPIRdist}; note that even when starting in $0$, the process will jump to one almost surely. This is not the case if $\varrho_i^{} = 0$; in this case, the process started in $0$ will stay there forever whence the convergence to  $\geom (\mathrm{e}^{-sT_i})$ then only holds for strictly positive $m_i^{}$. To see the last equality in Eq.~\eqref{135243}, we write out the exponential mixture of geometric distributions explicitly and substitute $x=\ee^{st}$ to obtain
\[
p_{i,\infty}(n_i^{}) = \int_0^\infty r_i^{} \ts \mathrm{e}^{-s t} (1-\mathrm{e}^{st} )^{n-1} r_i^{} \mathrm{e}^{-r_i^{} t}\dd t = \alpha \int_0^1 x^\alpha (1-x)^{n_i^{}-1} \dd x = \alpha B(n^{}_i,\alpha+1)
\]
as claimed.
\end{proof}

\begin{remark}\label{asymdegenerate}
In the degenerate case $\varrho_i^{} = r_i^{} = 0$ (where the YPIR is an ordinary Yule process), there is no stationary distribution because the  number of lines diverges almost surely. Nonetheless, one may  still define (somewhat informally) $p_{i, \infty}(n_i^{}) \defeq 0$ for all $n_i^{} \in \NN$ together with $p_{i,\infty} (\infty) = 1$. \hfill $\diamondsuit$
\end{remark}

Note that $p_{i, \infty}$ of \eqref{135243} is the \emph{Yule distribution} \cite{Yule,Simon} and may be rewritten as
\[
p_{i, \infty}(n_i^{}) = \alpha B(n_i^{},\alpha+1) = \frac{\alpha \Gamma(n_i^{}) \Gamma(\alpha+1)}{\Gamma(n_i^{}+\alpha+1)} = \frac{\alpha(n_i^{}-1)!}{\prod_{k=1}^{n_i^{}} (k+\alpha)}, \quad n_i^{}=1,2, \ldots,
\]
where $\Gamma$ denotes the gamma function.
In particular,  $p_{i,\infty}(1) = \alpha/ (1+\alpha)$. For $\alpha \gg 1$, therefore, $p_{i,\infty}$ is close to a point measure on 1; whereas for $\alpha \ll 1$, $p_{i,\infty}(n)$ is close to $\alpha/n_i^{}$ and puts substantial mass on large values, in line with intuition.

From the representation of the solution in Corollaries \ref{YPIRdist}--\ref{YPIRasymptotics} together with Eq.~\eqref{littleh} and $(1-x)^\infty=\delta_{x,0}$ for $x \in [0,1]$,  the  long-term behaviour of the solution is now immediate.
\begin{coro}\label{omegaunendlich}
Assuming that $\varrho_i^{} > 0$ for all $i \in S^*$, we have 
\begin{equation*}\label{formel fuer omegaunendlich}
\omega^{}_\infty \defeq \lim_{t \to \infty} \omega^{}_t =  \bigotimes_{i \in S} \pi^{}_i. \Big ( \big ( 1 - \gamma^{}_i (1 - f(\omega_0)) \big ) b(\omega_0) + \gamma^{}_i (1 - f(\omega_0)) d(\omega_0) \Big )
\end{equation*}
for all initial conditions $\omega_0^{} \in \cP(X)$.
As always, $f(\omega^{}_0)$ is the initial frequency of the beneficial type. Furthermore, $\gamma^{}_{i_\ast}(x) = \delta_{x,0}$ $($in line with Remark \textnormal{\ref{asymdegenerate}} and $(1-x)^\infty=\delta_{x,0}$$)$, and for $i \in S^*$,  $\gamma^{}_i$ is the probability generating function of $p_{i,\infty}$, that is,
\begin{equation*}
\gamma^{}_i(x) \defeq \sum_{n_i^{} = 1}^{\infty} p_{i,\infty}(n_i^{}) x^{n_i^{}}. 
\end{equation*}
\end{coro}

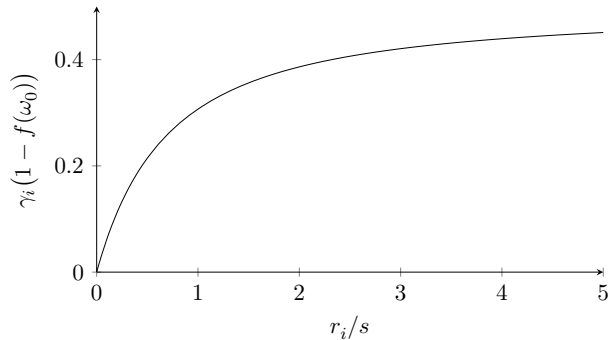
\begin{figure}[t]
\hspace{2cm}
\scalebox{0.8}{\parbox{\textwidth}{
\begin{tikzpicture}
\begin{axis}[
scaled ticks=false,
axis lines=left,
xlabel = $r^{}_i / s$,
ylabel = $\gamma_i \big (1 - f(\omega^{}_0) \big )$,
width = 10cm,
height = 6cm,
xmin = 0,
xmax = 5,
ymin = 0,
ymax = 0.5,
]
\addplot [black] table {selrec_plots_data.dat};
\end{axis}
\end{tikzpicture}
}}
\caption{\label{plot}
Asymptotic probability of site $i$ being drawn from $\pi^{}_i . d(\omega_0^{})$ as a function of  $r^{}_i/s$. As recombination becomes stronger, the asymptotic probability approaches the initial probability $1-f(\omega^{}_0)=1/2$ assumed here.}
\end{figure}

\begin{remark} From Corollary \ref{omegaunendlich}, it is clear that $\gamma^{}_i (1 - f(\omega_0))$ is the probability that site $i$ is
  drawn from $\pi^{}_i . d(\omega_0^{})$, or equivalently, that it is  associated with $i_*=1$ at equilibrium; see Figure~\ref{plot} for an illustration of its parameter dependence. For a site $i$ that is far away from  $i_\ast$ in the sense that its total rate of separation from $i_*$ is large in comparison to the selection strength ($s \ll r_i$),  the dynamics is close to that of the pure recombination equation; in particular, the marginals $\pi^{}_i . \omega^{}_t$ are approximately time invariant in line with the marginalisation consistency \eqref{constantletterfrequencies} of the pure recombination equation. Accordingly, the long-term behaviour is governed by $\gamma_i(x) \approx x$. In contrast, in the regime $s \gg r_i$,  the behaviour is closer to that of the pure selection equation in that $p_\infty$ places much weight  on large values, which  implies that $\gamma_i(x)$ is very small for small values of $x$, and the beneficial type prevails.
  \hfill $\diamondsuit$
\end{remark}

\appendix

\section{Marginalisation consistency}
\label{sec:marginals}
 
Let us consider the dynamics of the \emph{marginal type distributions} under selection and recombination. 
For $A \subseteq S$, we  define the \emph{marginal recombinators} $R^A_i : \cP(X_A) \to \cP(X_A)$ by
\begin{equation}\label{marginal recombinators}
R^A_i \nu \defeq \nu^{A \cap C_i} \otimes\nu^{A \cap D_i}
\end{equation}
for $i \in A \setminus i_*$, where $C_i$ and $D_i$ denote the head and tail for $i$ as before.

\begin{remark}\label{R=id}
Note  that $\pi^{}_{\! A} . R_i \omega = R_i^A \omega^A$ for all $A \subseteq S$ and $\omega \in \cP(X)$, and  $R_i^A = \id$ if $A$ is contained in either $C_i$ or $D_i$, that is,  if $\{C_i,D_i\} |_A = \{A\}$ (compare \cite[Lemma 1]{reco}).  \hfill $\diamondsuit$
\end{remark}

Consider now the marginal  $\omega^A$ of $\omega$.
It is  well known \cite[Proposition 6]{reco} that $\omega^A$  satisfies the \emph{marginalised recombination equation}
\begin{equation}\label{margrecoeq}
\dot{\omega}_t^A = \pi^{}_{\! A} . \Psireco(\omega^{}_t) = \sum_{i \in A \setminus i_*} \varrho_i^A \big ( R_i^A \omega_t^A - \omega_t^A \big ) \eqdef \Psireco^A(\omega_t^A)
\end{equation} 
with initial condition $\omega_0^A$ and \emph{marginal recombination rates}
\begin{equation}\label{margrecorates}
\varrho^A_i \defeq \!\!\! \!\!\! \sum_{\substack{j \in S^* \\ \{C_j,D_j\} |_A = \{C_i,D_i\} |_A}} \!\!\! \!\!\! \varrho_j \quad \text{for all } i \in A \setminus i_*;
\end{equation}
see Figure~\ref{marginal_scr} for an illustration. 
 In particular, 
\begin{equation}\label{constantletterfrequencies}
  \dot{\omega}_t^{\{i\}} = 0 \quad \text{for } i \in S^*
\end{equation}
since $R_i^{\{i\}}=\id$.
Eq.~\eqref{margrecoeq} follows from Remark \ref{R=id}, the linearity of $\pi^{}_{\! A} . $ and  Eq.~\eqref{margrecorates}.

\begin{figure}
\includegraphics[width=0.7\textwidth]{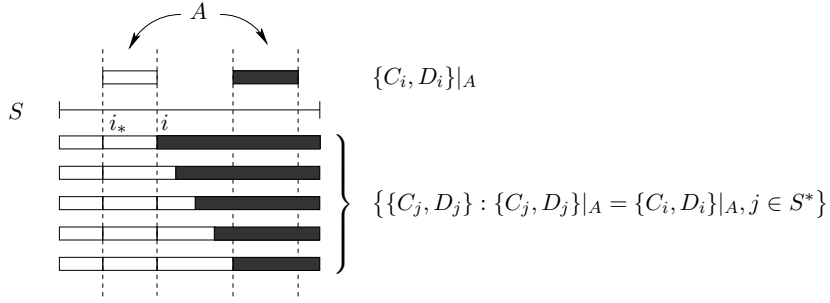}
\caption{\label{marginal_scr} The partitions in Eq.~\eqref{margrecorates} that define the marginal recombination rates. 
}
\end{figure}

Unfortunately, this property does not generalise to the selective case. The reason is that $\Psisel$  also depends on the proportion $f(\omega^{}_t)=\omega_t^{\{i_*\}}(0)$ of fit individuals and that we lose this information by projecting onto a factor with respect to a subset of S not containing  $i_\ast$. When $A$ \emph{does} contain $i_*$, however, we clearly have
\begin{equation}\label{fomega}
f(\nu) = f^A(\nu^A) \quad \text{for any } \nu \in \cP(X),
\end{equation}
where $f^A$ is defined analogously to  \eqref{fitproportion}, but with $S$ replaced by $A$.  Moreover, the selection operator  \eqref{fitnessoperator} acts consistently on subsystems that contain the selected site, that is,
$\pi_{\! A} . F \nu = F^A \nu^A \quad \text{for } A \ni i_*$,
where the marginalised selection operator is given by
$
F^A (\nu^A)(x^{}_{\! A}) \defeq (1 - x_{i_\ast}^{})   
\nu^A(x^{}_{\! A})
$.
We can thus define $\Psisel^A : \cP(X_A) \to \cP(X_A)$ via 
\begin{equation*}
\Psisel^A(\nu_A^{}) \defeq s \big (F^A - f^A(\nu_A^{}) \big) \nu_A^{}
\end{equation*}
such that
\[
\pi^{}_{\! A} . \Psisel(\nu)= \Psisel^A(\nu^A) \quad \text{for  } A \ni i^{}_* \quad \text{and  all } \nu \in \cP(X).
\]
Combining this with \eqref{margrecoeq}, we obtain the following result.

\begin{theorem}[marginalisation consistency of the SRE]\label{consistency}
Let $\omega$ be the solution of the initial value problem for the \textnormal{SRE} \eqref{main} with initial condition $\omega^{}_0$. Let $A \subseteq S$ contain $i_*$. Then, the marginal $\omega^A := (\omega_t^A)_{t \geq 0}$ solves the \emph{marginal SRE}, 
\begin{equation*}\label{mainmarginal}
\dot{\omega}^A_t  = 
s \big (F^A(\omega^A_t) - f^A(\omega^A_t) \omega^A_t \big ) + 
\sum_{i \in A \setminus i_*} \varrho_i^A \big ( R_i^A \omega_t^A - \omega_t^A \big ), 
\end{equation*}
with initial condition $\omega_0^A$ and \emph{marginal recombination rates}  \eqref{margrecorates}. 
In particular, $\omega^A$ is independent of all $\varrho_i$ with $i$ such that $\{C_i,D_i\} |_A = \{A\}$; or equivalently, with $i$ such that  $i  \succ j$ for all $j \in A$ comparable to $i$. \qed
\end{theorem}

\begin{remark}\label{inconsistency}
The problem of  marginalisation (in)consistency was already observed by Ewens and Thomson \cite{ET} in 1977
for the  discrete-time SRE; see also the
review in \cite[pp.~69--72]{Buerger}.
For Theorem~\ref{consistency}  to hold, the assumption that  $A$ contains the selected site is crucial:  It is otherwise impossible to find a closed expression for the projection of the selective part in \eqref{main} in terms of the marginal measure, because we lose the information about the proportion of fit individuals  in the case that $i_\ast \not \in A$. It is indeed a common pitfall to assume that Theorem \ref{consistency} holds for arbitrary $A$. This is also implicit in \cite{fehler}; see  the corresponding erratum.  \hfill $\diamondsuit$
\end{remark}

There is an interesting connection between marginalisation consistency and the recursive solution of the SRE of Theorem~\ref{rekursion}. Applying Theorem \ref{consistency}  to $A=\{i^{}_*\}$ shows that the marginal type frequency at the selected site is unaffected by recombination. More generally, consider the set
$L^{(k)} \defeq \{i_0 = i_\ast,i_1,\ldots,i_k\}$
and note that $L^{(k)} \setminus i_*$ is exactly the set of recombination sites that are considered up to and including the $k$-th iteration. 
Obviously, marginalisation consistency holds for $L^{(k)}$ for all $0 \leq k \leq n-1$. Since $\varrho_i^{L^{(k)}} = \varrho^{}_i$ for $i \in L^{(k)} \setminus i^{}_*$, Remark~\ref{R=id} and Eq.~\eqref{margrecorates} together with Definition~\ref{hierarchy} give
\[
 \pi^{}_{L^{(k)}} . \dot \omega^{}_t = \pi^{}_{L^{(k)}} .   \sum_{i \in L^{(k)} \setminus i^{}_*} \varrho^{}_i (R^{}_i \omega^{}_t - \omega^{}_t) =  \pi^{}_{L^{(k)}} .  \Psireco^{(k)} (\omega^{}_t) = \pi^{}_{L^{(k)}} . \dot \omega^{(k)}_t,
\]
and so $\pi_{L^{(k)}}^{} . \omega_t^{(k)} = \pi_{L^{(k)}}^{} . \omega_t^{}$. \
This implies that if one is only interested in the marginal with respect to $L^{(k)}$, then one may stop the iteration after the $k$th step. 

\section*{Acknowledgements}

It is a pleasure to thank M.~Baake for enlightening discussions and substantial suggestions to improve the manuscript, and  M.~ M\"ohle for an interesting exchange about duality. We also thank L.~Esercito and S.~Hummel for stimulating discussions, as well as an anonymous referee for helpful comments.  This
work was supported by the German Research Foundation (DFG), 
within the SPP 1590 and the CRC 1283 (project C1).

\end{document}